\documentclass[journal,twoside,web,twocolumn]{ieeecolor}
\usepackage{generic}
\usepackage{cite}
\usepackage{amsmath,amssymb,amsfonts}
\usepackage{algorithmic}
\usepackage{algorithm}
\usepackage{graphicx}
\usepackage{textcomp}

\usepackage{amsthm}

\newtheorem{assumption}{Assumption}

\newtheorem{lemma}{Lemma}
\newtheorem{thm}{Theorem}
\newtheorem{remark}{Remark}
\newtheorem{corollary}{Corollary}

\DeclareMathOperator*{\argmin}{argmin}  
\definecolor{red1}{RGB}{192,0,0}

\def\BibTeX{{\rm B\kern-.05em{\sc i\kern-.025em b}\kern-.08em
    T\kern-.1667em\lower.7ex\hbox{E}\kern-.125emX}}
\begin{document}
\title{Decentralized Composite Optimization in Stochastic Networks: A Dual Averaging Approach with Linear Convergence}
\author{Changxin Liu, \IEEEmembership{Member, IEEE}, Zirui Zhou, Jian Pei, \IEEEmembership{Fellow, IEEE}, Yong Zhang, and Yang Shi, \IEEEmembership{Fellow, IEEE}
	\thanks{This work was supported by the Natural Sciences and Engineering Research Council of Canada. \textit{(Corresponding author: Yang Shi.)} }

\thanks{C. Liu and Y. Shi are with the Department of Mechanical Engineering, University of Victoria, BC V8W 3P6, Canada (e-mail: changxin@kth.se; yshi@uvic.ca).}
\thanks{Z. Zhou and Y. Zhang are with the Vancouver Research Center, Huawei, Burnaby, BC V5C 6S7, Canada (e-mail: zirui.zhou@huawei.com; yong.zhang3@huawei.com).}
\thanks{J. Pei is with the School of Computing Science, Simon Fraser University, Burnaby, BC V5A 1S6, Canada (e-mail: jpei@cs.sfu.ca).}
}

\maketitle

\begin{abstract}
Decentralized optimization, particularly the class of decentralized composite convex optimization (DCCO) problems, has found many applications. Due to ubiquitous communication congestion and random dropouts in practice, it is highly desirable to design decentralized algorithms that can handle stochastic communication networks. However, most existing algorithms for DCCO only work in networks that are deterministically connected during bounded communication rounds, and therefore cannot be extended to stochastic networks.
In this paper, we propose a new decentralized dual averaging (DDA) algorithm that can solve DCCO in stochastic networks. Under a rather mild condition on stochastic networks, we show that the proposed algorithm attains \emph{global linear convergence} if each local objective function is strongly convex.
Our algorithm substantially improves the existing DDA-type algorithms as the latter were only known to converge \emph{sublinearly} prior to our work. The key to achieving the improved rate is the design of a novel dynamic averaging consensus protocol for DDA, which intuitively leads to more accurate local estimates of the global dual variable. To the best of our knowledge, this is the first linearly convergent DDA-type decentralized algorithm and also the first algorithm that attains global linear convergence for solving DCCO in stochastic networks. Numerical results are also presented to support our design and analysis.
\end{abstract}


\section{Introduction}
\label{sec:introduction}
Consider a group of $n$ agents (e.g., processors, machines), each of which has its own objective function. They are connected via a bidirectional communication network and aim to cooperatively solve the following convex composite optimization problem in a decentralized manner:
\begin{equation}\label{OPT}
	\min_{x\in\mathbb{R}^m}  \left\{ F(x) := \frac{1}{n}\sum_{i=1}^{n}f_i(x) +h(x) \right\},
\end{equation}
where
$f_i$ is the local smooth objective function of agent $i$ and $h$ is a non-smooth regularization term that is shared across all the agents. Problem~\eqref{OPT} is referred to as decentralized convex composite optimization \cite{shi2015proximal,alghunaim2019linearly} and finds broad applications in optimal control of multi-agent systems~\cite{raffard2004distributed}, resource allocation~\cite{uribe2019resilient}, and large-scale machine learning~\cite{lian2017can}, just to name a few \cite{yang2019survey}.

In this work, we focus on solving Problem \eqref{OPT} when the communication network is \emph{stochastic}. There are many practical reasons that promote the consideration of stochastic communication networks. Indeed, communication in real networks is usually subject to congestion, errors, and random dropouts, which is typically modeled as a stochastic process. Besides, stochastic networks are useful for proactively reducing communication cost. For instance, the gossip protocol~\cite{boyd2006randomized} and Bernoulli protocol~\cite{kar2008sensor}, which randomly choose a subset of communication links from an underlying dense graph in each iteration, have been widely regarded as effective strategies to avoid high communication cost and network congestion. Therefore, it is highly desirable to develop decentralized algorithms that solve Problem~\eqref{OPT} over stochastic communication networks and attain a favorable convergence rate.

Over the past decade, many algorithms have been proposed for solving Problem~\eqref{OPT}. Some of them exploit the composite structure in~\eqref{OPT} and attain global \emph{linear} convergence if Problem~\eqref{OPT} is strongly convex (see, e.g.,~\cite{latafat2019new,alghunaim2019linearly}), which is the fastest rate of convergence that one can expect from a first-order decentralized algorithm. However, such linear convergence results are limited to \emph{time-invariant} communication networks, because the design of these algorithms inherently requires knowledge of network topology \emph{a priori}. Indeed, these algorithms are typically developed upon leveraging centralized primal-dual optimization paradigms, such as the alternating direction method of multipliers (ADMM)~\cite{boyd2010distributed}, to solve the following problem that is equivalent to~\eqref{OPT}:
\begin{equation}\label{eq:trans-prob}
	\min_{x_1,\dots,x_n\in\mathbb{R}^m}  \ \frac{1}{n}\sum_{i=1}^{n}\Big( f_i(x_i)+h(x_i) \Big) \quad
	\mbox{s.t.} \  (\mathcal{L}\otimes I){\mathbf x}=0,
\end{equation}
where $\mathbf{x} = [x_1^T, \dots, x_n^T]^T$, $\otimes$ denotes the Kronecker product, $I$ is an identity matrix of size $m\times m$, and $\mathcal{L}$ denotes the graph Laplacian associated with the communication network. Since $\mathcal{L}$ needs to be explicitly given in formulation~\eqref{eq:trans-prob}, these algorithms and their associated linear convergence results cannot be extended to stochastic communication networks, where the network topology is time-varying and random.
Among the existing decentralized optimization methods, the decentralized dual averaging (DDA) algorithm proposed by~\cite{NIPS2010_faa9afea} and its later extensions~\cite{tsianos2012push,lee2016coordinate,pmlr-v48-colin16} have been recognized as a powerful framework that can handle stochastic networks. However, the convergence rates of existing DDA-type algorithms are rather slow. In fact, even for decentralized convex \emph{smooth} optimization in time-invariant networks, which is deemed to be much simpler than Problem \eqref{OPT} in stochastic networks, these algorithms were only known to converge \emph{sublinearly}. Specifically, existing DDA-type algorithms, when applied to Problem~\eqref{OPT}, only attain an $\mathcal{O}({1}/{\sqrt{t}})$ sublinear rate of convergence. For the special case of Problem~\eqref{OPT} with $h\equiv 0$,~\cite{liu2020towards} recently showed that the convergence rate can be improved to $\mathcal{O}(1/t)$. Nevertheless, it remains open whether a DDA-type algorithm can attain linear rate of convergence.

\textbf{Contribution.} In this paper, we propose a new DDA algorithm that solves Problem \eqref{OPT} in stochastic networks. Under a rather mild condition on the stochastic network, we show that the proposed algorithm has an $\mathcal{O}(1/t)$ rate of convergence in the general case and a global linear rate of convergence if each local objective function is strongly convex. Our work contributes to the literature of decentralized optimization in the following two aspects:
\begin{itemize}
	\item[i)] We develop the first decentralized algorithm that attains global linear convergence for solving Problem\eqref{OPT} in stochastic networks. Existing linearly convergent decentralized algorithms for Problem \eqref{OPT} only work in networks that are deterministically connected during bounded communication rounds, and therefore cannot be extended to stochastic networks.
	Our algorithm is based on a DDA framework that is fundamentally different from these algorithms.
	
	\item[ii)] Our algorithmic design and convergence analysis shed new light on DDA-type algorithms. Notably, it is the first DDA-type algorithm that attains linear convergence. Prior to our work, even for decentralized convex \emph{smooth} optimization in time-invariant networks, existing DDA-type algorithms were only known to converge sublinearly. The key to achieving the improved rate is the design of a novel dynamic averaging consensus protocol for DDA, which intuitively leads to more accurate local estimates of the global dual variable.
\end{itemize}

\section{Related Works}

{\bf Decentralized algorithms for Problem \eqref{OPT} in deterministic networks.} Due to its broad applications, Problem \eqref{OPT} has received attention in the community of decentralized optimization for many years; see, e.g., \cite{shi2015proximal} for an early attempt. It is only until recently that linearly convergent decentralized algorithms have been developed for solving Problem \eqref{OPT} in determinisitc networks.
For time-invariant networks, \cite{alghunaim2019linearly} developed a decentralized proximal gradient method, where the diffusion step and the proximal step are designed differently from~\cite{shi2015proximal} such that not only the fixed point meets the global optimality condition but also linear convergence can be attained for strongly convex problems. {Furthermore, the strategy was generalized as a unified framework for proximal gradient tracking in \cite{alghunaim2020decentralized}.}
\cite{latafat2019new} proposed a distributed algorithm based on randomized block-coordinate proximal method, which exhibits an asymptotic linear convergence if the monotone operator associated with Problem~\eqref{OPT} is metrically subregular (a much weaker condition than strong convexity).
Very recently,~\cite{xu2020unified} proposed a unified decentralized algorithmic framework based on the operator splitting theory, which attains linear convergence for the strongly convex case. 
For deterministic time-varying networks, the authors in \cite{sun2019distributed} developed a linearly convergent decentralized optimization algorithm based on the gradient-tracking technique and elaborate objective surrogates.
However, it still requires the network to be connected during bounded communication rounds, which is a worst-case assumption about network connectivity \cite{lobel2010distributed} and does not necessarily hold in stochastic networks.
To summarize, existing linearly convergent decentralized algorithms for Problem \eqref{OPT} are only applicable to deterministic networks and cannot be extended to the stochastic networks, which motivates the new algorithm development and convergence analysis in this paper.

{\bf Decentralized optimization in stochastic networks.}
The study of decentralized algorithms over stochastic networks dates back to~\cite{lobel2010distributed}, who proposed a subgradient-based algorithm with diminishing step sizes. The decentralized dual averaging algorithm, which combines dual averaging method~\cite{nesterov2009primal} and consensus-seeking, was reported by~\cite{NIPS2010_faa9afea} and can handle stochastic networks with an $\mathcal{O}(1/\sqrt{t})$ sublinear rate of convergence.
The decentralized accelerated gradient algorithm with a random network model was proposed by~\cite{jakovetic2013convergence}, where an $\mathcal{O}(\frac{\log t}t)$ sublinear convergence rate is obtained for smooth problems.
{A decentralized ADMM algorithm was designed in \cite{chang2016proximal}, where a few nodes are randomly selected to perform local updates.
Decentralized optimization with asynchronous local updates was considered in  \cite{peng2016arock,bastianello2020asynchronous}.}
Later,~\cite{xu2017convergence,pu2020push} validated the use of a constant step size in decentralized gradient descent over stochastic networks, leading to a global linear rate of convergence for strongly convex and smooth problems. Recently, \cite{koloskova2020unified} developed a unified framework for decentralized stochastic gradient descent over stochastic networks.
It is worth mentioning that the aforementioned studies either consider general non-smooth problems or focus on smooth problems. In particular, they cannot exploit the composite structure of Problem~\eqref{OPT}, partially due to the technical difficulty caused by the so-called \emph{projection-consensus coupling} \cite{NIPS2010_faa9afea} for methods integrating consensus-seeking and projected/proximal gradient descent. 

In summary, to the best of our knowledge, no existing methods can solve or can be easily extended to solve Problem \eqref{OPT} in stochastic networks with global linear convergence.

\section{Preliminaries}\label{sec:preliminaries}
\subsection{Basic Setup}
We consider the finite-sum optimization problem~\eqref{OPT}, in which $h$ is a closed convex function with its domain, denoted by $\mbox{dom}(h)$, being non-empty, and $f_i$ satisfies the following assumptions for all $i=1,\dots,n$. {Typical choices of $h$ include the elastic net regularization, i.e., $h(x)=\lambda_1 \lVert x \rVert_1+\lambda_2\lVert x \rVert_2^2, \lambda_1,\lambda_2\geq 0$, and the indicator function of a closed convex set.}
\begin{assumption}\label{problemassumption}
	i) $f_i$ is continuously differentiable on an open set that contains $\mbox{dom}(h)$;
	ii) $f_i$ is (strongly) convex with modulus $\mu\geq 0$ on $\mbox{dom}(h)$, i.e., for any $x,y\in\mbox{dom}(h)$, 
	\begin{equation}
		\label{eq:strongly-convex}
		f_i(x) - f_i(y) - \langle\nabla f_i(y), x - y \rangle \geq \frac{\mu}{2}\|x - y\|^2;
	\end{equation}
	and iii) $\nabla f_i$ is Lipschitz continuous on $\mbox{dom}(h)$ with Lipschitz constant $L>0$, i.e., for any $x,y\in\mbox{dom}(h)$,
	\begin{equation}
		\label{eq:Lip-origin}
		\|\nabla f_i(x) - \nabla f_i(y)\| \leq L\|x - y\|.
	\end{equation}
\end{assumption}
Throughout the paper, we denote by $x^*$ an optimal solution of Problem~\eqref{OPT}. Assumption~\ref{problemassumption} is standard in the study of decentralized optimization \cite{nedic2017achieving,shi2015proximal}. It is worth noting that we allow $\mu = 0$ in Assumption~\ref{problemassumption}(ii), which reduces to the general convex case.
\subsection{Stochastic Communication Networks}
We consider solving Problem~\eqref{OPT} in a decentralized manner, that is, each agent $i$ holds a local objective function $F_i := f_i + h$ and a pair of agents can exchange information only if they are connected in the communication network. Similar to existing studies \cite{nedic2009distributed,NIPS2010_faa9afea,shi2015extra,shi2015proximal}, we use a doubly stochastic matrix $P^{(t)}\in[0,1]^{n\times n}$ to encode the network topology and the weights of connected links at time $t$. We focus on the fairly general setting of stochastic communication networks, i.e., $P^{(t)}$ is a random matrix for every $t$. For the convergence of the proposed decentralized algorithm, we make the following assumption on $P^{(t)}$.

\begin{assumption}\label{graphconnected}
	For every $t\geq 0$, it holds that
	{i) the network is undirected};
	ii) 
	$P^{(t)}\mathbf{1} = \mathbf{1}$ and $\mathbf{1}^TP^{(t)} = \mathbf{1}^T$, where $\mathbf{1}$ denotes the all-one vector of dimensionality $n$;
	iii) $P^{(t)}$ is independent of the random events that occur up to time $t-1$; and
	iv) there exists a constant $\beta\in(0,1)$ such that
	\begin{equation}
		\label{eq:sigma-bound}
		\sqrt{\rho\left(\mathbb{E}_t\left[{P^{(t)}}^TP^{(t)}\right] - \frac{\mathbf{1}\mathbf{1}^T}{n}\right)} \leq \beta,
	\end{equation}
	where $\rho(\cdot)$ denotes the spectral radius and the expectation $\mathbb{E}_t[\cdot]$ is taken with respect to the distribution of $P^{(t)}$ at time $t$.
\end{assumption}

{Assumption~\ref{graphconnected} has been used for analyzing the convergence of a host of decentralized algorithms; see, e.g.,~\cite{boyd2006randomized,xu2017convergence,koloskova2020unified}. 
It is satisfied by numerous stochastic communication  settings;
we take the following two common settings as examples. 
\textit{i) Randomized gossip}: At every time $t$ one communication link $(i,j)$ is sampled from an underlying graph $\mathcal{G}$. Suppose that we take $P^{(t)} = I-\frac{1}{2}(e_i-e_j)(e_i-e_j)^T$, where $I$ is the identity matrix and $e_i\in\mathbb{R}^n$ is a vector with $1$ in the $i$-th position and $0$ otherwise. Then, it is known that Assumption~\ref{graphconnected} is satisfied provided that the underlying graph $\mathcal{G}$ is connected; see, e.g.,~\cite{boyd2006randomized}. \textit{ii) Bernoulli stochastic networks}: Consider an underlying graph $\mathcal{G}$, where the state (online or offline) of each link $(i,j)$ is a Bernoulli process with link probability $w_{ij}$.
Suppose the corresponding Bernoulli processes are statistically independent  for different pairs of edges, and $w_{ij}=w_{ji}$.
Denote by $\mathcal{L}^{(t)}$ the Laplacian at time $t$, and set $P^{(t)} = I - \mathcal{L}^{(t)}/(2d)$, where $d = \max_i d_i$ and $d_i$ is the degree of node $i$ in $\mathcal{G}$. It can be verified that Assumption 2 holds when the second largest eigenvalue of Laplacian average $\overline{\mathcal{L}}$ is strictly positive \cite{kar2008sensor}.} 




\subsection{Dual Averaging Method}

Our algorithm is based on the dual averaging method that was originally proposed by~\cite{nesterov2009primal}. The dual averaging method originally proposed by~\cite{nesterov2009primal} can be directly applied to solve Problem~\eqref{OPT} in a \emph{centralized} manner. In particular, let $d$ be a strongly convex function with modulus $1$ on $\mbox{dom}(h)$ such that 
\begin{equation}
	\label{eq:initial-cond}
	x^{(0)} = \argmin_{x\in\mathbb{R}^m} d(x) \in \mbox{dom}(h) \ \ \mbox{and} \ \ d(x^{(0)}) = 0.
\end{equation}
Then, the dual averaging method starts with $x^{(0)}$ and iteratively generates $\{x^{(t)}\}_{t\geq 1}$ according to
\begin{equation}
	\label{eq:standard-DA}
	x^{(t)} = \argmin_{x\in\mathbb{R}^m} \left\{ \sum_{\tau = 0}^{t - 1} a_{\tau + 1} \ell(x; x^{(\tau)}) + d(x) \right\},
\end{equation}
where 
\begin{equation}
	\label{eq:at}
	a_t = \frac{a}{(1 - a\mu)^t}, \quad t = 1,2,\dots
\end{equation}
for some constant $a>0$, $\ell:\mathbb{R}^m\times\mathbb{R}^m\rightarrow\mathbb{R}$ is defined as
\begin{equation}
	\label{eq:def-ell}
	\ell(y;z) := f(z) + \langle \nabla f(z), y - z\rangle + \frac{\mu}{2}\|y - z\|^2 + h(y)
\end{equation}
for any $y,z\in\mathbb{R}^m$, and $f = \frac{1}{n}\sum_{i=1}^nf_i$. It is worth noting that for the strongly convex case (i.e., $\mu>0$), the sequence $\{a_t\}_{t\geq 1}$ is geometrically increasing; for the general convex case (i.e., $\mu = 0$), the sequence $\{a_t\}_{t\geq 1}$ equals the constant $a$. Moreover, both~\eqref{eq:standard-DA} and~\eqref{eq:at} require the modulus $\mu$ of strong convexity. In practice, one can use a lower bound of $\mu$ or simply set $\mu = 0$ in~\eqref{eq:standard-DA} and~\eqref{eq:at} if no valid lower bound is available.

The following theorem summarizes the convergence property of the above dual averaging method, which is a direct extension of Theorem 3.2 by~\cite{lu2018relatively} to problems with non-smooth regularization terms.  A proof of Theorem~\ref{exact_thm} is provided in \cite[Appendix F]{liu2021decentralized}.

\begin{thm}\label{exact_thm}
	When Assumption~\ref{problemassumption} is satisfied, let $\{x^{(t)}\}_{t\geq 0}$ be the sequence of iterates generated by the dual averaging method~\eqref{eq:standard-DA}, if $a \leq L^{-1}$, then
	\begin{equation*}
		F(\tilde{x}^{(t)}) - F(x^*) \leq \frac{d(x^*)}{A_t}, \quad t = 1,2,\dots,
	\end{equation*}
	where $A_t = \sum_{\tau = 1}^ta_\tau$ and $\tilde{x}^{(t)} = A_t^{-1}\sum_{\tau=1}^ta_\tau x^{(\tau)}$. Moreover, the following estimates on $A_t^{-1}$ hold: i) If $\mu > 0$, then 
	$
	\frac{1}{A_t} \leq \frac{(1 - a\mu)^t}{a}
	$; and
	ii) If $\mu = 0$, then
	$
	\frac{1}{A_t} = \frac{1}{at}.
	$
\end{thm}

\section{Algorithm and Main Results}\label{sec:alg-main}
\label{sec:main-results}

From Theorem~\ref{exact_thm}, one can observe that the dual averaging method, when applied to solve Problem~\eqref{OPT} in a \emph{centralized} manner, attains global linear convergence if Problem~\eqref{OPT} is strongly convex. The existing dual averaging based \emph{decentralized} algorithms, however, converge only sublinearly. In view of this, the following question arises naturally: can we develop a dual averaging based decentralized algorithms that can achieve the same order of convergence as its centralized counterpart, that is, linear convergence? In this section, we put an affirmative answer to this question by developing a new DDA algorithm that incorporates a novel dynamic averaging consensus protocol for each local update, which intuitively leads to more accurate local estimates of the global dual variable. We show that the new DDA, when applied to solve Problem~\eqref{OPT} in stochastic networks, converges linearly if each local objective is strongly convex. Our algorithmic design and convergence analysis shed new light on DDA-type algorithms, as it is the first DDA-type algorithm that can achieve linear convergence. Besides, it is also the first linearly convergent algorithm for solving Problem \eqref{OPT} in stochastic networks.


To motivate the design of our DDA method, we observe that by letting $A_t = \sum_{\tau = 1}^t a_{\tau}$ and
$$ z^{(t)} = \sum_{\tau = 0}^{t-1}a_{\tau+1}\left(\frac{1}{n}\sum_{i=1}^n\nabla f_i(x^{(\tau)}) - \mu x^{(\tau)}\right), $$
the update rule~\eqref{eq:standard-DA} can be written as
\begin{equation}\label{transfored_DA}
	 x^{(t)} = \argmin_{x\in\mathbb{R}^m}\left\{\langle z^{(t)}, x\rangle + A_t\left(\frac{\mu}{2}\|x\|^2 + h(x) \right) + d(x)\right\}. 
\end{equation}
Thus, it is sensible for each agent to locally estimate the global dual variable $z^{(t)}$ to fulfill decentralization. To this end, we propose the following dynamic averaging consensus protocol:
\begin{subequations}\label{2nd_order_consensus}
	\begin{align}
		z_i^{(t)} &=\sum_{j=1}^np_{ij}^{(t-1)}\left(z_{j}^{(t-1)}+ a_{t}s_{j}^{(t-1)}\right), \label{z}  \\
		s_{i}^{(t)} &= \sum_{j=1}^np_{ij}^{(t-1)}s_{j}^{(t-1)}+\left(\nabla f_i(x_{i}^{(t)})-\mu x_{i}^{(t)}\right) \nonumber \\
		&\quad \quad \quad \quad \quad \quad \quad\quad  -\left(\nabla f_i(x_{i}^{(t-1)})-\mu x_{i}^{(t-1)}\right) \label{s},
	\end{align}
\end{subequations}
where $p_{ij}^{(t)}$ is the $(i,j)$-th element in the mixing matrix $P^{(t)}$, $z_i^{(t)}$ is the $i$-th agent's local estimate of $z^{(t)}$ at time $t$ and $s_i^{(t)}$ is an auxiliary vector for reducing consensus error. Equipped with these, each agent $i$ can perform a local computation to update its estimate of the global primal variable $x^{(t)}$: 
\begin{equation}\label{primal}
	\begin{split}
		{x}_{i}^{(t)}
		= \argmin_{x\in\mathbb{R}^m}\Big\{ \langle z_{i}^{(t)},x \rangle+ A_{t}\Big(\frac{\mu}{2}\lVert x\rVert^2 +h(x)\Big )+d(x)\Big\} .
	\end{split}	
\end{equation}
We denote by $\mathcal{N}_i^{(t)}$ the set of agents that are connected with agent $i$ at time $t$. 
Then, the entire algorithm can be summarized in Algorithm~\ref{LinDDA}.

\begin{algorithm}[tb]
	\caption{The proposed decentralized dual averaging algorithm for Problem~\eqref{OPT}}
	\label{LinDDA}
	\begin{algorithmic}[1]
		\STATE {\bfseries Input:} $\mu\geq 0$, $a>0$, $x^{(0)}\in\mbox{dom}(h)$ and a strongly convex function $d$ with modulus $1$ on $\mbox{dom}(h)$ such that~\eqref{eq:initial-cond} holds
		\STATE {\bfseries Initialize:} $a_0 = a$, $A_0 = 0$, $x_i^{(0)} = x^{(0)}$, $z_i^{(0)} = 0$, and $s_i^{(0)} = \nabla f_i(x^{(0)})-\mu x^{(0)}$ for all $i = 1, \dots, n$ 
		\FOR{$t=1,2,\cdots$}
		\STATE set $a_{t}=a_{t-1}/(1 - a\mu)$ and $A_{t}=A_{t-1}+a_{t}$
		\STATE \emph{In parallel (for agent $i$, $i = 1,\dots,n$)}
		\STATE collect $z_{j}^{(t-1)}$ and $s_{j}^{(t-1)}$ from all agents $j\in\mathcal{N}_{i}^{(t-1)}$
		\STATE update $z_{i}^{(t)}$ and $s_{i}^{(t)}$ by~\eqref{2nd_order_consensus}
		\STATE compute $x_i^{(t)}$ by~\eqref{primal}
		\STATE broadcast $z_{i}^{(t)}$ and $s_{i}^{(t)}$ to all agents $j\in\mathcal{N}_{i}^{(t)}$
		\ENDFOR
	\end{algorithmic}
\end{algorithm}

Our protocol~\eqref{2nd_order_consensus} differs from the one used in the original DDA \cite{NIPS2010_faa9afea} in~\eqref{s}, where the latter simply lets $s_i^{(t)} = \nabla f_i(x_i^{(t)})$ for all agents $i$. Our update in~\eqref{s} is a second order dynamic averaging consensus protocol motivated by~\cite{zhu2010discrete}, and equips each agent $i$ with an $s_i^{(t)}$ that can track the global variable $\frac{1}{n}\sum_{i=1}^n(\nabla f_i(x_i^{(t)})-\mu x_i^{(t)})$. Intuitively,~\eqref{2nd_order_consensus} can lead to much more accurate local estimates $\{z_i^{(t)}\}_{i=1}^n$ of the global dual variable $z^{(t)}$. Moreover, as we will show in the convergence analysis, the novel update~\eqref{s} validates the use of \emph{geometrically increasing} weights $\{a_t\}_{t\geq0}$ in \eqref{z}. This contrasts with the use of \emph{decaying} weights in other DDA-type algorithms and is key to achieving the linear convergence result. 


Before proceeding, we make some remarks on Algorithm~\ref{LinDDA}. First, Algorithm~\ref{LinDDA} provides a unified treatment for both general convex and strongly convex cases. In particular, if $\mu = 0$, we simply set $a_t = a$ and $A_t = at$ for all $t$. Second, to satisfy the condition in~\eqref{eq:initial-cond}, one can choose an arbitrary $x^{(0)}\in\mbox{dom}(h)$ and let
$ d(x) := \tilde{d}(x) - \tilde{d}(x^{(0)}) - \langle \nabla \tilde{d}(x^{(0)}), x - x^{(0)}\rangle, $
where $\tilde{d}$ is any strongly convex function with modulus $1$, e.g., $\tilde{d}(x) = \|x\|^2/2$. It is easy to verify that such $x^{(0)}$ and $d$ satisfy~\eqref{eq:initial-cond}.
{Third, for agent $i$ with $\mathcal{N}_i^{(t-1)}=\emptyset$, it will not perform Step 6. However, all the agents are required to compute $x_i^{(t)}$ according to Steps 7 and 8. }
 {Finally, similar to the standard dual averaging method, we assume that the subproblem~\eqref{primal} can be computed easily. This holds for a host of applications. For example, if we choose $d(x) = \|x - x^{(0)}\|^2/2$, then the subproblem~\eqref{primal} reduces to computing the proximal operator of $A_th/(1+\mu A_t)$, which admits a closed-form solution in many applications. 
Compared to ADMM-based methods \cite{shi2014linear}, where typically a non-trivial dual problem is solved at each iteration, the proposed method has lighter computational cost per step. 
When subproblem~\eqref{primal} cannot be computed efficiently, one may run another loop to compute an approximate solution, which is common decentralized composite optimization.}

\begin{remark} (Intuition behind \eqref{2nd_order_consensus})
In the centralized dual averaging update \eqref{transfored_DA}, only $z^{(t)}$ contains global information. Therefore, if $z^{(t)}$ can be estimated sufficiently accurate by the agents, they can solve \eqref{OPT} in a decentralized way. We follow the idea in \cite{zhu2010discrete} that the second-order dynamic average consensus can be used to estimate the average of local signals whose second-order differences are relatively bounded. Particularly, observe that 
\begin{equation*}
	z_i^{(t)} = \sum_{\tau=0}^{t-1}a_{t+1}\left( \nabla f_i(x_i^{(\tau)})-\mu x_i^{(\tau)}\right).
\end{equation*} 
Take $\nabla f_i(x_i^{(\tau)})-\mu x_i^{(\tau)}-\left( \nabla f_i(x_i^{(\tau-1)})-\mu x_i^{(\tau-1)}\right)$ as the second-order difference of $z_i^{(t)}$. By \eqref{eq:Lip-origin}, one obtains its upper bound
\begin{equation*}
\begin{split}
&\left\lVert  \nabla f_i(x_i^{(\tau)})-\mu x_i^{(\tau)}-\left( \nabla f_i(x_i^{(\tau-1)})-\mu x_i^{(\tau-1)}\right) \right\rVert \\
&\leq (L+\mu)\lVert x_i^{(\tau)} - x_i^{(\tau-1)}\rVert, \forall i=1,\cdots, n.
\end{split}
\end{equation*}
Therefore, the dynamic average consensus scheme can be used to estimate $n^{-1}\sum_{i=1}^n z_i^{(t)}$ which is a good approximation of $z^{(t)}$. This motivates our update formulas \eqref{2nd_order_consensus}.
\end{remark}

\begin{remark}(Comparison with gradient-tracking methods \cite{xu2017convergence,qu2017harnessing})
Recall the update in \cite{xu2017convergence,qu2017harnessing} as follows
\begin{equation}\label{gradient-tracking}
	\begin{split}
	x_i^{(t)} &= \sum_{j=1}^n p_{ij}^{(t-1)}\left( x_j^{(t-1)}-a s_j^{(t-1)}\right) \\
    s_i^{(t)} &= \sum_{j=1}^n p_{ij}^{(t-1)} s_j^{(t-1)} + \nabla f_i(x_i^{(t)}) - \nabla f_i(x_i^{(t-1)})
	\end{split}
\end{equation}
It contains three key differences from our update in \eqref{2nd_order_consensus}: \textbf{i}) \eqref{gradient-tracking} updates $s_i^{(t)}$ by estimating the average of local gradients $n^{-1}\sum_{i=1}^n\nabla f_i(x_i^{(t)})$, but \eqref{2nd_order_consensus} updates $s_i^{(t)}$ by estimating the term $n^{-1}\sum_{i=1}^n \left( \nabla f_i(x_i^{(\tau)})-\mu x_i^{(\tau)}\right)$; \textbf{ii}) \eqref{gradient-tracking} weights $s_i^{(t)}$ with constant $a$, but \eqref{2nd_order_consensus} weights $s_i^{(t)}$ with geometically increasing $\{a_t\}_{t\geq 0}$ that is motivated by the dual averaing method; \textbf{iii}) In the proposed DDA method, another proximal operator is performed over $z_i^{(t)}$ to get $x_i^{(t)}$. The gradient evaluated over $x_i^{(t)}$ is then used to update $s_i^{(t)}$. While \eqref{gradient-tracking} does not accommodate proximal operators.
\end{remark}

The rest of this section presents the convergence results of Algorithm~\ref{LinDDA}. To proceed, we denote 
\begin{equation}\label{eq:matrix-M}
	\mathbf{M}=\begin{bmatrix}
		\beta & \beta \\
		\frac{a(L+\mu)}{1-a\mu}\Big( \beta +\frac{1}{1-a\mu} \Big) & \frac{\beta+{a\beta(L+\mu)}}{1-a\mu}
	\end{bmatrix},
\end{equation}
where $L$ and $\mu$ are given in Assumption~\ref{problemassumption}, $\beta\in(0,1)$ is defined in Assumption~\ref{graphconnected}, and $a$ is an input of Algorithm~\ref{LinDDA}. Matrix $\mathbf{M}$ is the key to our convergence analysis as it defines the dynamics of the iterates generated by Algorithm~\ref{LinDDA}. 
Let $\rho(\mathbf{M})$ be the spectral radius of $\mathbf{M}$. To facilitate the presentation of our convergence analysis, we define
\begin{equation}
	\label{eq:def-nu}
		\begin{aligned}
	\nu  &:= \rho(\mathbf{M})\sqrt{1 - a\mu}, \\
	\eta  &:= (1 - a\mu)(1 - \nu)^2, \\
	\theta  &:= (1 - a\mu)(1 - \nu^2).
		\end{aligned}
\end{equation}
The following result on $\nu$, $\eta$, and $\theta$ is fundamental to our convergence analysis whose proof can be found in Appendix~\ref{sec:proof-Lemma1}.

\begin{lemma}
	\label{lem:M-radius}
	The value of $\nu$ monotonically increases with $a$ if $a\in(0, 1/\mu)$. 
	Moreover, if 
	\begin{equation}
		\label{eq:condition1-a}
		\frac{1}{a} > \frac{\beta(2L + 3\mu)}{(1 - \beta)^2} + \mu,
	\end{equation}
	then $\nu < 1$. Consequently, $\eta$ and $\theta$ are both positive and monotonically decrease with $a$ if~\eqref{eq:condition1-a} is satisfied.
\end{lemma}

Equipped with Lemma~\ref{lem:M-radius}, we are ready to present the main results of this paper, which pertain to the convergence property of Algorithm~\ref{LinDDA}. Similar to some existing works \cite{NIPS2010_faa9afea}, we first present the convergence property of an auxiliary sequence $\{y^{(t)}\}_{t\geq 0}$, which then immediately implies the convergence property of the sequence $\{x_i^{(t)}: i = 1,\dots,n\}_{t\geq 0}$ generated by Algorithm~\ref{LinDDA}. In particular, we define
\begin{equation}
	\label{eq:def-y}
	y^{(t)}
	= \argmin_{x\in\mathbb{R}^m}\left\{  \langle \overline{z}^{(t)},x \rangle+ A_{t}\left(\frac{\mu }{2}\lVert x\rVert^2+h(x)\right)+d(x)\right\},
\end{equation}
where $y^{(0)}=x^{(0)}$, $\overline{z}^{(t)} = \frac{1}{n}\sum_{i=1}^nz_i^{(t)}$ and $\{z_i^{(t)}: i = 1,\dots, n\}_{t\geq 0}$ are generated by Algorithm~\ref{LinDDA}.

\begin{thm}\label{inexact_thm}
	When Assumptions~\ref{problemassumption} and~\ref{graphconnected} are satisfied, the constant $a$ in Algorithm~\ref{LinDDA} satisfies~\eqref{eq:condition1-a}, and
	\begin{align}
		\gamma := \frac{1}{a} - 2L + {\mu} - \frac{4L-2\mu}{\eta} & > 0, \label{2ndcondition}
	\end{align}
	where $\eta$ is defined in~\eqref{eq:def-nu}, then, for all $t\geq 1$, it holds that
	\begin{equation}\label{convergence}
		\begin{split}
			\mathbb{E}[F(\tilde{y}^{(t)})]  -F(x^*)
			\leq  \frac{C}{A_t},
		\end{split}
	\end{equation}
	where $\tilde{y}^{(t)}= A_t^{-1}\sum_{\tau=1}^{t}a_\tau y^{(\tau)}$ with $y^{(\tau)}$ defined in~\eqref{eq:def-y},
	\begin{equation*}
		\begin{split}
			C
			:=	{d(x^*)+ \frac{a\big(2L-\mu\big) \sigma^2}{n\theta(L+\mu)^2}} > 0,
		\end{split}
	\end{equation*} 
	and $\sigma^2$ is the variance of local gradients at $t = 0$, i.e.,
	$$ \sigma^2 = \sum_{i=1}^n\left\|\nabla f_i(x^{(0)}) - \frac{1}{n}\sum_{j=1}^n\nabla f_j(x^{(0)})\right\|^2. $$
	Moreover, for all $t\geq 1$ and $i = 1,\dots,n$, we have
	\begin{equation}
		\label{eq:x-y-dist}
		\mathbb{E}[\|\tilde{x}_i^{(t)} - \tilde{y}^{(t)}\|^2] \leq \frac{D}{A_t},
	\end{equation}
	where $\tilde{ x}_{i}^{(t)}=A_t^{-1}\sum_{\tau=1}^{t}a_{\tau}{ x}_{i}^{(\tau)}$ and
	$ D = \frac{4nC}{\eta\gamma} + \frac{2a\sigma^2}{\theta(L+\mu)^2} > 0. $
\end{thm}  

The proof of Theorem~\ref{inexact_thm} is postponed to Appendix~\ref{appendix:pf-main}. Theorem~\ref{inexact_thm} can be regarded as a decentralized counterpart of Theorem~\ref{exact_thm}. Due to the presence of consensus error in the decentralized setting, Theorem~\ref{inexact_thm} requires a more delicate choice of $a$ for convergence. 
It is shown in \cite[Appendix E]{liu2021decentralized} that there exists an $\bar{a}>0$ such that any $a\in(0,\bar{a})$ satisfies the conditions in Theorem~\ref{inexact_thm}. Moreover, $\bar{a}$ is roughly in the order $\mathcal{O}((1-\beta)^2/L)$. 


As a consequence of Theorem~\ref{inexact_thm}, we show in Corollary \ref{coro:1} that Algorithm~\ref{LinDDA} attains global linear convergence if $\mu>0$. Its proof is given in Appendix~\ref{appendix:pf-coro1}.
\begin{corollary}\label{coro:1}
	Suppose that the premise of Theorem~\ref{inexact_thm} holds. If $\mu > 0$, then for all $t\geq 1$ and $i = 1,\dots,n$, we have
	\begin{equation}
		\label{eq:linear-conv}
		\mathbb{E}[\|\tilde{x}_i^{(t)} - x^*\|^2] \leq \frac{2}{a}\left(\frac{2C}{\mu} + D\right)(1 - a\mu)^{t},
	\end{equation}
	where $\tilde{ x}_{i}^{(t)}=A_t^{-1}\sum_{\tau=1}^{t}a_{\tau}{ x}_{i}^{(\tau)}$ and $C, D$ are positive constants given in Theorem~\ref{inexact_thm}.
\end{corollary}
To the best of our knowledge, Corollary~\ref{coro:1} provides the first linear convergence result for solving Problem \eqref{OPT} in stochastic networks. It is also the first linear convergence result for any DDA-type algorithms.

{
\begin{remark}
	Corollary~\ref{coro:1} requires $h$ to be uniform for all agents. From a technical perspective, this is necessarily made to ensure that (38) in the proof of Theorem 2 and Lemma 5 remain valid. The authors in \cite{alghunaim2020decentralized} proved that with agent specific non-smooth regularization, linear convergence cannot be achieved (in the worst case) for decentralized composite optimization even in time-invariant networks.
\end{remark}}

As Corollary~\ref{coro:1} also holds when Algorithm~\ref{LinDDA} is applied to solving Problem~\eqref{OPT} with $h\equiv 0$ in time-invariant networks, it would be interesting to compare our linear convergence result with those of decentralized algorithms that also converge linearly in this special case. Based on the above remark on $a$, one can observe that the rate of linear convergence in Corollary~\ref{coro:1} is roughly $\mathcal{O}((1-\rho)^t)$, where $\rho = \Theta\left((1-\beta)^2/\kappa\right)$. This rate is sub-optimal for decentralized convex smooth optimization in time-invariant networks, where a better rate $\rho = \Theta\left((1-\beta)/\kappa\right)$ is achieved by, for example,~\cite{xu2020unified}. 
{This is mainly due to the consensus-based gradient-tracking mechanism.
DIGing \cite{nedic2017achieving}, Harnessing \cite{qu2017harnessing}, and AugDGM \cite{xu2017convergence} are based on similar strategies, for which the convergence rate is $\Theta((1-\beta)^2/\kappa)$. Indeed, under the algorithmic framework based on operator splitting \cite{xu2020unified}, the gradient-tracking mechanism essentially leads to a contraction matrix $I-(I-P)^2$ on the recursion.
The convergence rates of decentralized primal-dual algorithms, e.g., EXTRA \cite{shi2015extra}, NIDS \cite{li2019decentralized}, typically have better dependence on the network topology, that is, $\Theta((1-\beta)/\kappa)$. However, they are limited to time-invariant networks to the best of our knowledge.}



For the case $\mu = 0$, Theorem~\ref{inexact_thm} implies that Algorithm~\ref{LinDDA} has a global $\mathcal{O}(1/t)$ rate of convergence. In particular, we have the following corollary whose proof is presented in Appendix~\ref{sec:proof-corollary2}.
\begin{corollary}\label{coro:2}
	Suppose that the premise of Theorem~\ref{inexact_thm} holds. If $\mu = 0$, then for all $t\geq 1$ and $i = 1,\dots,n$, we have
	\begin{align}
		\mathbb{E}[F(\tilde{y}^{(t)})]  -F(x^*) & \leq  \frac{C}{at}, \label{eq:y-sublinear} \\
		\mathbb{E}[\|\tilde{x}_i^{(t)} - \tilde{y}^{(t)}\|^2] & \leq \frac{D}{at}, \label{eq:x-to-y}
	\end{align}
	where $\tilde{y}^{(t)}=\frac{1}{t}\sum_{\tau=1}^{t}y^{(\tau)}$, $\tilde{x}_{i}^{(t)}=\frac{1}{t}\sum_{\tau=1}^{t}{x}_{i}^{(\tau)}$, and $C,D$ are positive constants given in Theorem~\ref{inexact_thm}. In addition, if $h \equiv 0$ in Problem~\eqref{OPT}, $d(x) = \|x\|^2/2$, and
	\begin{equation}
		\label{x_condition_smooth}
		\begin{split}
			\frac{1}{a}&> 2L \cdot \max \left\{  \frac{\beta}{(1-\beta)^2} , 1 +\frac{6}{(1-\nu)^2}\right\},
		\end{split}
	\end{equation}
	where $\beta$ and $\nu$ are given in~\eqref{eq:sigma-bound} and~\eqref{eq:def-nu}, respectively,
	then we further have
	\begin{equation}
		\label{eq:x-h=0}
		\mathbb{E}[F(\tilde{x}_{i}^{(t)})]-F(x^*)
		\leq \frac{1}{t}\left( \frac{n\lVert x^* \rVert^2}{2a}+ \frac{6 \sigma^2}{L\big(1-\nu^2\big)}\right).
	\end{equation}
\end{corollary}
Similar to some existing works (e.g.,~\cite{wai2017decentralized}), we can only ensure the $\mathcal{O}(1/t)$ rate for the objective value at the auxiliary sequence $\{\tilde{y}^{(t)}\}_{t\geq 1}$ and the distance of each agent's local estimate $\tilde{x}_i^{(t)}$ to $\tilde{y}^{(t)}$ when $h\not\equiv 0$; see~\eqref{eq:y-sublinear} and~\eqref{eq:x-to-y} respectively. The major difficulty is that we cannot derive $\frac{1}{n}\sum_{i=1}^{n}{x}_i^{(t)} = y^{(t)}$ when $h \not\equiv 0$, which prevents us from getting~\eqref{convergence:xtoy}. It remains open whether the $\mathcal{O}(1/t)$ rate for the objective value at $\{\tilde{x}_i^{(t)}\}_{t\geq 1}$, as in~\eqref{eq:x-h=0}, can be established when $h\not\equiv 0$ without additional assumptions. We leave it as future work.

\section{Numerical Experiments}
\label{sec:numerical}

For the experiments, we consider the decentralized LASSO problem \cite{wai2017decentralized} and the decentralized sparse logistic regression problem~\cite{alghunaim2019linearly}.
We present numerical results of Algorithm~\ref{LinDDA} (named as DDA below), and compare it with the following algorithms:

\textit{ i)} Proximal gradient exact first-order algorithm (PG-EXTRA) in \cite{shi2015proximal}:
	\begin{equation*}
		\begin{split}
			\mathbf{z}^{(t)} =& \mathbf{z}^{(t-1)}-\mathbf{x}^{(t-1)}+\tilde{\mathbf{P}}(2\mathbf{x}^{(t-1)} -\mathbf{x}^{(t-2)} )\\
			&-a(\nabla^{(t-1)}-\nabla^{(t-2)}), \\
			\mathbf{x}^{(t)} =& \mathrm{Prox}_{\mathbf{h}}^a(\mathbf{z}^{(t)}),
		\end{split}
	\end{equation*}
	where $\tilde{\mathbf{P}}=\frac{(I+P)\otimes I}{2}$,  $\mathbf{h}(\mathbf{x})=\sum_{i=1}^{n}h(x_i)$, and
$
		\mathrm{Prox}_{\mathbf{h}}^a(\mathbf{z}):=\argmin_{\mathbf{x}\in\mathbb{R}^{mn}} \left\{  \mathbf{h}(\mathbf{x})+\frac{1}{2a}\lVert \mathbf{x}-\mathbf{z} \rVert^2  \right\}.
$
	
	\textit{ii)} Proximal primal-dual diffusion (P2D2) algorithm in \cite{alghunaim2019linearly}:
	\begin{equation*}
		\begin{split}
			\mathbf{z}^{(t)} =& \left(I-\alpha\mathbf{B}\right)\mathbf{z}^{(t-1)}+\left(I-\mathbf{B}\right)(\mathbf{x}^{(t-1)} -\mathbf{x}^{(t-2)} )\\
			&-a(\nabla^{(t-1)}-\nabla^{(t-2)}),\\
			\mathbf{x}^{(t)} =& \mathrm{Prox}_{\mathbf{h}}^a(\mathbf{z}^{(t)}),
		\end{split}
	\end{equation*}
	where $\mathbf{B}=\frac{(I-P)\otimes I}{2}$.
	
	\textit{iii)} Distributed subgradient method (DSM) in \cite{lobel2010distributed}:
	\begin{equation*}
		\begin{split}
			\mathbf{x}^{(t)} &= \mathbf{P}^{(t-1)}\mathbf{x}^{(t-1)}-a_{t-1}\mathbf{r}^{(t-1)},\\
		\end{split}
	\end{equation*}
	where $\mathbf{r}^{(t)} \in \partial \mathbf{F}(\mathbf{x}^{(t)})$ and $\mathbf{F}(\mathbf{x})=\sum_{i=1}^{n}F(x_i)$.
	
	\textit{iv)} Conventional DDA (named as C-DDA below) in \cite{duchi2011dual}:
	\begin{equation*}
		\begin{split}
			\mathbf{z}^{(t)} =& \mathbf{P}^{(t-1)}\mathbf{z}^{(t-1)}+\mathbf{r}^{(t-1)},\\
			\mathbf{x}^{(t)} =& \argmin_{\mathbf{x}\in\mathbb{R}^{mn}}  \left\{ a_{t-1}\langle 	\mathbf{z}^{(t)} , 	\mathbf{x}  \rangle+ \mathbf{d}(\mathbf{x})\right \},
		\end{split}
	\end{equation*}
	where $\mathbf{d}(\mathbf{x})=\sum_{i=1}^{n}d(x_i)$.

We note that when applied to solve Problem~\eqref{OPT} in stochastic networks, PG-EXTRA and P2D2 have no convergence guarantees and DSM and C-DDA have sublinear convergence in theory.

\subsection{Strongly Convex Problems}
The aforementioned algorithms are applied to the following problem:
\begin{equation*}
	\min_{x\in\mathbb{R}^m}\frac{1}{n}\sum_{i=1}^{n}\frac{1}{2}\lVert b_i-C_i x  \rVert^2,\quad \mathrm{s.t.}\,\, \lVert x \rVert_1\leq R,
\end{equation*}
where $R>0$ is a constant, and $(C_i, b_i)$ represents the data tuple available to agent $i$ with $C_i\in \mathbb{R}^{60\times 50}$ and $b_i\in \mathbb{R}^{60}$.
The data is randomly generated according to the setting by~\cite{li2019decentralized}.
Firstly, a sparse signal $x^\sharp \in \mathbb{R}^{50}$ is randomly generated, where the probability for each element being nonzero is $0.25$. 
Then, each $C_i$ is randomly generated and then normalized such that Assumption~\ref{problemassumption} holds with $L=1$ and $\mu=0.5$. 
Set $R=1.1\lVert x^\sharp \rVert_1$.
produced based on $b_i = C_i x^\sharp+\epsilon_i$, where $\epsilon_i$ is a random noise vector.


We consider two common configurations of stochastic communication networks. The first one is \emph{Bernoulli networks} \cite{kar2008sensor}, where a fixed graph is first generated and at any time $t$, each edge of the fixed graph is sampled with probability $\iota\in(0,1)$, which results in a random sub-graph of the fixed graph. In our experiment, we generate a fixed graph in the same way as~\cite{shi2014linear}, where the sparsity parameter $\xi$, i.e., the ratio between the number of edges in the generated fixed graph and the number of edges in the complete graph, is chosen to be $0.2$. Based on each fixed graph, we generate two Bernoulli networks with $\iota$ set to be $0.05$ and $0.1$, respectively. The second one is \emph{randomized gossip networks} \cite{boyd2006randomized}, where only a single edge of a fixed graph is sampled at any time $t$. 
In particular, 
the probability to sample the link $(i,j)$ is set as $\frac{1}{n(\lvert \mathcal{N}_i \rvert+1)}$ with $\lvert \mathcal{N}_i \rvert$ representing the number of neighbors of $i$ in the supergraph at every time $t$.
In our experiment, we consider cycle graph, 2D grid, and complete graph as the fixed graphs for generating randomized gossip networks.

For all the tested algorithms, we evaluate their performance in terms of the relative square error (RSE) defined by $\frac{\sum_{i=1}^{n}\lVert x_i^{(t)}-x^*\rVert^2}{\sum_{i=1}^{n}\lVert x_i^{(0)}-x^*\rVert^2}$, where $x^*$ is identified by applying the centralized proximal gradient method~\cite{parikh2014proximal} to Problem \eqref{eq:exp} such that the norm of the difference of two consecutive iterates is less than $10^{-14}$. The algorithm by~\cite{condat2016fast} is used to perform projection onto $l_1$-norm ball.
All the algorithms are initialized with $x_{i}^{(0)}=0$ for all $i$.
The parameters for each algorithm are chosen in the following way. For DDA and C-DDA, we employ $d(x) = \lVert  x\rVert^2/2$. We choose $\alpha=0.5$ in P2D2 and set $a_t ={1}/{\sqrt{t+1}}$ for C-DDA. For the two groups of Bernoulli networks, we set the $a$ in DDA to be $0.1$ and
set $0.1$ for the step sizes in P2D2 and PG-EXTRA. For randomized gossip, we use $a=0.1$ for DDA, and set $10^{-4}$ for the step sizes in P2D2 and PG-EXTRA. Since DSM can not be applied to constrained problems, it is not considered in this setting.

The simulation results are plotted in Figure~\ref{fig:LASSO}.
In particular, the performance on Bernoulli networks and randomized gossip networks is presented in the first and the second column of Figure \ref{fig:LASSO}, respectively.
In the first column, the bottom plot demonstrates the performance in time-invariant networks that is used for generating Bernoulli networks. 
Although P2D2 and PG-EXTRA demonstrate a similar performance with DDA on time-invariant networks, they do not converge to the minimizer when applied to stochastic networks.
In line with our theoretical results, DDA linearly converges and outperforms C-DDA in all the network configurations.


\begin{figure}
	\begin{center}
		\includegraphics[width=0.48\textwidth]{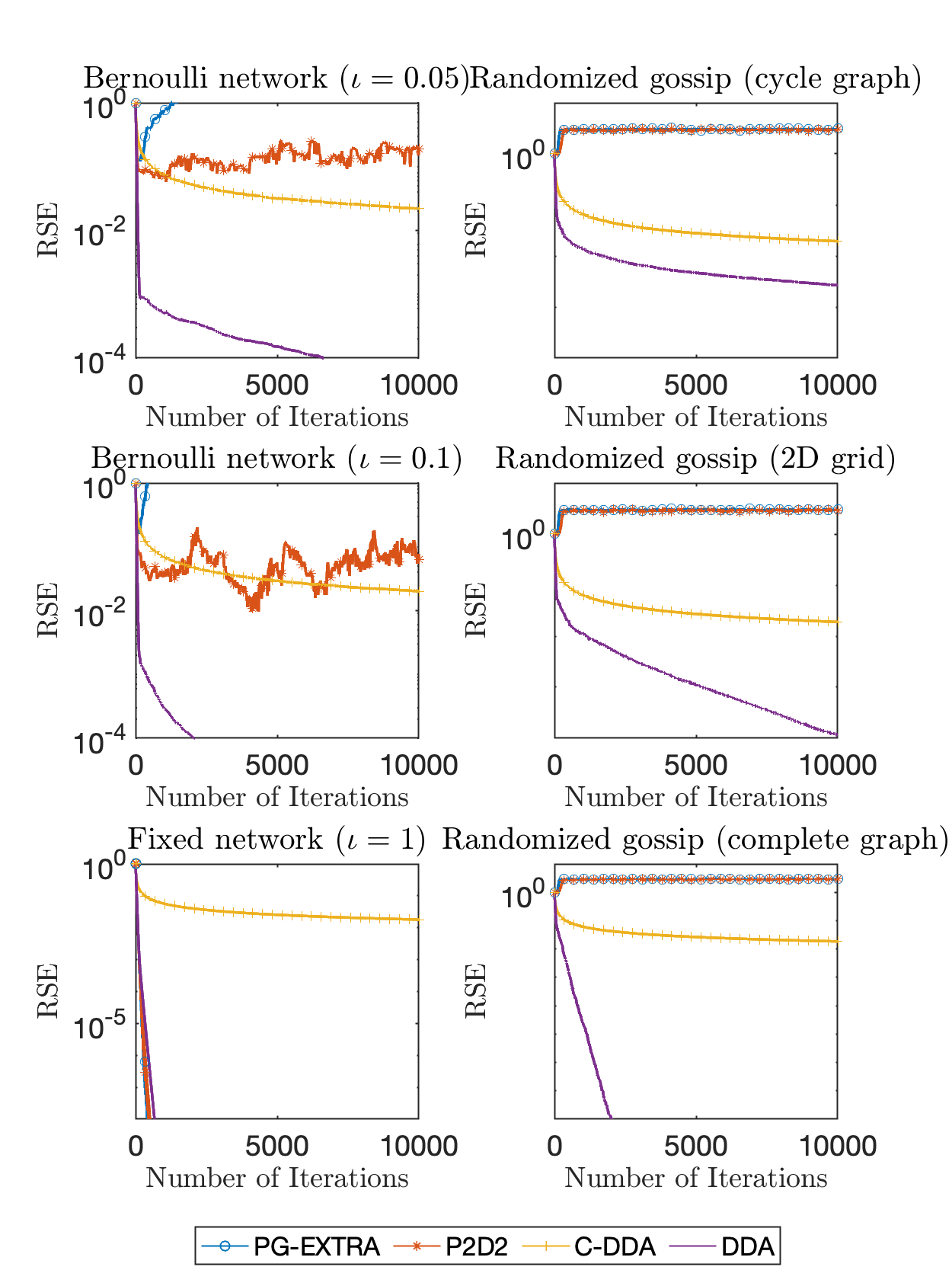}    
		\caption {Comparison results for decentralized LASSO in different network configurations.}
		\label{fig:LASSO}
	\end{center}
\end{figure}


\subsection{General Convex Problems}
{
The following decentralized sparse logistic regression problem is considered
\begin{equation}\label{eq:exp}
	\min_{x\in\mathbb{R}^m}\frac{1}{n}\sum_{i=1}^{n}f_i(x)+\phi\lVert x \rVert_1,
\end{equation}
where
\begin{equation*}
	f_i(x)=\frac{1}{m_i}\sum_{j=1}^{m_i}\ln (1+\exp(-y_j^i M_j^{i\mathrm{T}} x)),
\end{equation*}
and $\{M_{j}^i, y_{j}^i \}_{j=1}^{m_i}$ are data samples private to agent $i$. In our experiment, we set $\phi=0.001$, and use \emph{Spambase} data set in the UCI Machine Learning Repository~\cite{Dua:2019} to generate our problem instance. In particular, we extract $3000$ out of the total $4601$ samples in the original data set and evenly distribute them to the $n = 30$ agents, i.e., $m_i = 100$ for all $i$. }

Two types of stochastic communication networks are considered. For Bernoulli networks, we generate a fixed graph with the sparsity parameter $0.4$. Based on it, we construct two Bernoulli networks by setting $\iota=0.1$ and $\iota=0.2$, respectively. In the second setting, we also consider cycle graph, 2D grid, and complete graph as the fixed graphs for generating randomized gossip networks.
We identify $x^*$ by using the centralized proximal gradient method, where the stopping criterion is set as the norm of the difference of two consecutive iterates smaller than $10^{-14}$. 
The performance of all the tested algorithms is evaluated in terms of the suboptimality defined by $F(n^{-1}\sum_{i=1}^{n}x_i^{(t)}) - F(x^*)$.
All the algorithms are initialized with $x_{i}^{(0)}=0$ for all agents $i$. 
The parameters of each algorithm are chosen properly to reflect their performance. For DDA and C-DDA, we simply choose $d(x) = \lVert  x\rVert^2/2$. We choose $\alpha=0.5$ in P2D2 and set $a_t ={1}/{\sqrt{t+1}}$ for C-DDA and DSM. 
For the Bernoulli networks, i.e., those sampled from a supergraph with sparsity parameter $0.4$ (first column of Figure \ref{fig:Bernoulli}), we use the same $a = 0.2$ for DDA, P2D2, and PG-EXTRA.
For randomized gossip, we use $a=0.05$ for DDA, and set $10^{-4}$ for the step sizes in P2D2 and PG-EXTRA. We note that choosing a smaller step size in P2D2 and PG-EXTRA generally makes them more stabilizing. In fact, a larger step size will result in even worse behaviour of these two methods in randomized gossip networks.

\begin{figure}
	\begin{center}
		\includegraphics[width=0.48\textwidth]{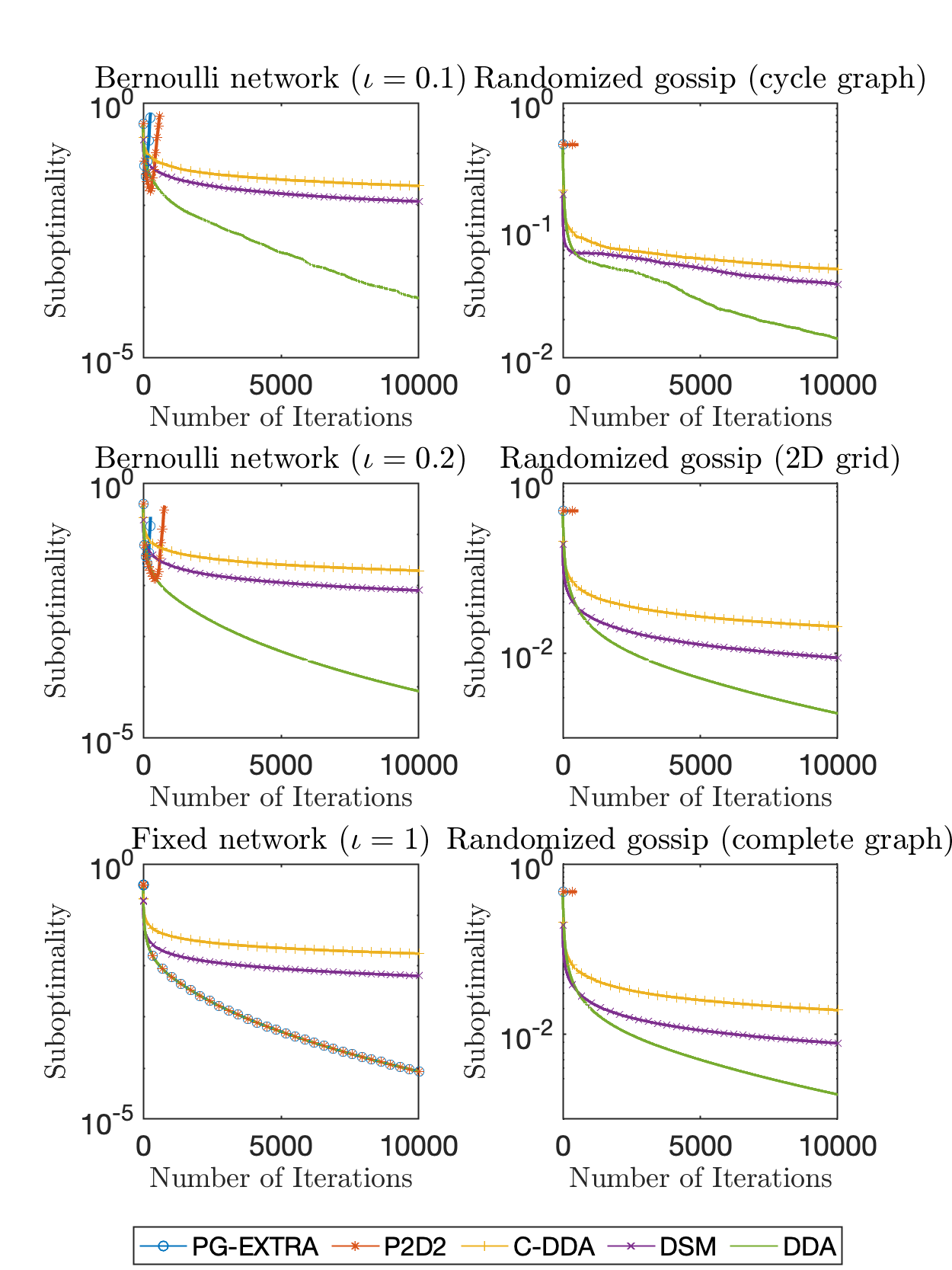}    
		\caption {Comparison results for decentralized logistic regression in different network configurations.}
		\label{fig:Bernoulli}
	\end{center}
\end{figure}

The simulation results are plotted in Figure~\ref{fig:Bernoulli}. Specifically, the first column of Figure \ref{fig:Bernoulli} presents the performance on Bernoulli networks and the second column shows the performance on randomized gossip networks. We note that the last plot of the first column is for time-invariant networks, which is the fixed graph for generating the Bernoulli networks in the first column.
One can observe that our DDA is substantially faster than DSM and C-DDA in all the network settings, which supports our theoretical development. In addition, while P2D2 and PG-EXTRA perform very similar to DDA on time-invariant networks, they both diverge when applied to stochastic networks. This suggests that decentralized algorithms that are designed for time-invariant networks may not work effectively in stochastic networks.

To summarize, the simulation results confirm our theoretical findings and demonstrate the superior performance of the proposed Algorithm~\ref{LinDDA} on both time-invariant and stochastic networks.

\section{Conclusion and Future Work}\label{sec:conclusion}
In this paper, we proposed a new decentralized algorithm for solving Problem \eqref{OPT} in stochastic networks. The proposed algorithm, based on the framework of dual averaging method, is facilitated by designing a novel dynamic averaging consensus protocol. To the best of our knowledge, this is the first linearly convergent DDA-type decentralized algorithm and also the first algorithm that attains global linear convergence for solving Problem \eqref{OPT} in stochastic networks.

As we remarked after Corollary \ref{coro:1}, it remains open whether Algorithm \ref{LinDDA} can be further improved such that it achieves the optimal rate of linear convergence when applied to the special setting of Problem \eqref{OPT}. Besides, it is unknown if $\mathcal{O}(1/t)$ convergence rate of the objective error at the local estimates can be established for the general convex case. {Another practical issue about decentralized optimization is the privacy risk due to information exchange among multiple agents}. We leave them as future research.

\appendices
\section{Overview of the Appendix and Preliminaries}
\label{appendix:notations}

In the Appendix, we begin with some notation to streamline the presentation. Then, we present the proof of Theorem~\ref{inexact_thm} and the proofs of Corollaries~\ref{coro:1} and~\ref{coro:2} in Appendix~\ref{appendix:pf-main} and \ref{appendix:pf-coro}, respectively. The proofs of supporting lemmas are postponed to Appendix \ref{supporting_lemmas}. 

First, we introduce the following notation:
\begin{equation}\label{eq:notation-1}
	\begin{split}
	&{\bf x}^{(t)}=\begin{bmatrix}
		x_{1}^{(t)}\\ \vdots \\ x_{n}^{(t)}
	\end{bmatrix}, \quad 
	{\bf s}^{(t)}=\begin{bmatrix}
		s_{1}^{(t)}\\ \vdots \\ s_{n}^{(t)}
	\end{bmatrix}, \quad 
	{\bf z}^{(t)}=\begin{bmatrix}
		z_{1}^{(t)}\\ \vdots \\ z_{n}^{(t)}
	\end{bmatrix},\\
	&{\bf \nabla}^{(t)}=\begin{bmatrix}
		\nabla f_1(x_{1}^{(t)})\\ \vdots \\ \nabla f_n(x_{n}^{(t)})
	\end{bmatrix}, \quad 
	{\bf y}^{(t)}=\begin{bmatrix}
		y^{(t)}\\ \vdots \\ y^{(t)}
	\end{bmatrix}, \quad  \overline{z}^{(t)}=\frac{1}{n}\sum_{i=1}^{n}z_i^{(t)},
\end{split}
\end{equation}
\begin{equation}\label{eq:notation-2}
	\overline{x}^{(t)} = \frac{1}{n}\sum_{i=1}^n x_i^{(t)}, \,\, \overline{g}^{(t)} = \frac{1}{n}\sum_{i=1}^n \nabla f_i(x_i^{(t)}), \,\, \overline{s}^{(t)}=\frac{1}{n}\sum_{i=1}^{n}s_i^{(t)}, 
\end{equation}
\begin{align}
	\tilde{\bf s}^{(t)} = {\bf s}^{(t)}-\mathbf{1}\otimes \overline{s}^{(t)}, \quad \tilde{\bf z}^{(t)} = {\bf z}^{(t)}-\mathbf{1}\otimes\overline{z}^{(t)}, \label{eq:notation-3} 
\end{align}
\begin{align}
	 \Delta{ \overline{x}}^{(t-1)} = {\overline{x}}^{(t)}-{ \overline{x}}^{(t-1)},\quad \Delta y^{(t-1)} = y^{(t)} - y^{(t-1)},
\end{align}
\begin{align}
	\Delta\mathbf{x}^{(t-1)} = \mathbf{x}^{(t)} - \mathbf{x}^{(t-1)}, \quad \Delta\mathbf{y}^{(t-1)} = \mathbf{y}^{(t)} - \mathbf{y}^{(t-1)}, \label{eq:notation-4}
\end{align}
where $\mathbf{1}$ is an all-one column vector of dimension $n$. We remark that bold lowercase letters represent a vector of dimension $m\times n$, while normal lowercase letters represent a vector of dimension $m$. Equipped with these notation, we can re-write the update rule~\eqref{2nd_order_consensus} in the following compact form:
\begin{subequations}
	\begin{align}
		{\bf z}^{(t)}&=  \mathbf{P}^{(t-1)}\Big({\bf z}^{(t-1)} + a_{t}{\bf s}^{(t-1)}\Big), \label{concise_2nd_consensus} \\
		{\bf s}^{(t)}&=\mathbf{P}^{(t-1)}{\bf s}^{(t-1)}+\nabla^{(t)}-\nabla^{(t-1)}-\mu \Delta{ \bf x}^{(t-1)}\label{consensus_on_s},
	\end{align}
\end{subequations}
where $\mathbf{P}^{(t)}= P^{(t)}\otimes I$ with $I$ being an identity matrix of size $n\times n$.

For a real-valued random vector $x$, we define
\begin{equation}\label{expectation_operator}
	\lVert x \rVert_{\mathbb{E}} = \sqrt{  \mathbb{E} [\lVert x \rVert^2]}.
\end{equation}
Accordingly, for a square random matrix $W$, we define
$
\lVert W \rVert_{\mathbb{E}} =\sup_{\lVert x \rVert_{\mathbb{E}}=1}\lVert Wx \rVert_{\mathbb{E}}
$.
Given two real-valued random vectors $x,y$, 
the Minkowski inequality~\cite{gut2013probability} states that
\begin{equation}\label{Minkowski}
	\lVert  x +y \rVert_{\mathbb{E}} \leq \lVert  x \rVert_{\mathbb{E}}+\lVert  y \rVert_{\mathbb{E}}.
\end{equation}
Finally, it is known that Assumption \ref{problemassumption}(iii) implies that 
\begin{equation}
	\label{eq:Lip-cond}
	f_i(x) - f_i(y) - \left\langle\nabla f_i(y), x - y \right\rangle \leq \frac{L}{2}\|x - y\|^2
\end{equation}
for any $x,y\in\mbox{dom}(h)$. 

\section{Proof of Theorem~\ref{inexact_thm}}
\label{appendix:pf-main}
In this section, we provide the proof of Theorem~\ref{inexact_thm}. 

To start, we show the following result that quantifies the deviation between the local estimates $\{x_i^{(t)}\}_{t\geq 0}$ and the auxiliary sequence $\{y^{(t)}\}_{t\geq 0}$.
\begin{lemma}\label{consensus_error_lemma}
	Suppose that $a$ satisfies~\eqref{eq:condition1-a}. Then, for all $t\geq 0$, it holds that
	\begin{equation}
		\begin{split}
		\label{consensus_error}
		&\sum_{\tau=0}^{t}{a_{\tau+1}}\mathbb{E}[\lVert{{\bf x}}^{(\tau)}-{{\bf y}}^{(\tau)}\rVert^2]  \\
		& \leq
		\frac{2}{\eta}\sum_{\tau=0}^{t-1}{a_{\tau+1}}\mathbb{E}[\lVert \Delta{\bf y}^{(\tau)}\rVert^2] +  \frac{2a\sigma^2}{\theta(L+\mu)^2},		
	\end{split}
	\end{equation}
	where $\sigma$ is defined in Theorem~\ref{inexact_thm}, $\eta$ and $\theta$ are given in~\eqref{eq:def-nu}, and both $\eta$ and $\theta$ are positive due to~\eqref{eq:condition1-a} and Lemma~\ref{lem:M-radius}.
\end{lemma}   
\begin{proof}[Proof of Lemma~\ref{consensus_error_lemma}]
	The proof is postponed to Appendix~\ref{sec:proof-Lemma2}.
\end{proof}
Lemma~\ref{consensus_error_lemma} states that if $a$ satisfies~\eqref{eq:condition1-a}, then
the accumulative deviation between ${\bf y}^{(t)}$ and ${\bf x}^{(t)}$ admits an upper bound constituted by the successive change of ${\bf y}^{(t)}$ plus a constant. 

Next, we present the following lemma that pertains to a descent-like property of Algorithm~\ref{LinDDA}.
\begin{lemma}
	\label{lem:descent}
	For all $t\geq 1$, it holds that
	\begin{align}\label{lem:inexact-DA}
		& \sum_{\tau=1}^ta_{\tau}\left(\left\langle \overline{g}^{(\tau-1)}, y^{(\tau)} - x^*\right\rangle + h(y^{(\tau)}) - h(x^*)\right)-   d(x^*) \nonumber \\
		&\leq \frac{\mu}{2}\sum_{\tau=1}^ta_{\tau}\left(\|\overline{x}^{(\tau-1)} - x^*\|^2 - \|\overline{x}^{(\tau - 1)} - y^{(\tau)}\|^2\right)\nonumber \\
		&\quad  -\frac{1}{2}\sum_{\tau=1}^t(1+\mu A_{\tau-1})\|y^{(\tau)} - y^{(\tau-1)}\|^2. 
	\end{align}
\end{lemma}
\begin{proof}[Proof of Lemma~\ref{lem:descent}]
	The proof is postponed to Appendix~\ref{sec:proof-Lemma3}.
\end{proof}

Equipped with the above two technical lemmas, we are ready to present the proof of Theorem~\ref{inexact_thm}.
\begin{proof}[Proof of Theorem~\ref{inexact_thm}]
	For all $\tau\geq 0$, one has
	\begin{align}
		& \frac{1}{n}\sum_{i=1}^{n}a_\tau\left( f_i({y}^{(\tau)}) -f_i(x^*)\right) \nonumber \\
		& \leq \frac{1}{n}\sum_{i=1}^{n} a_\tau\Big( f_i(x_{i}^{(\tau-1)}) - f_i(x^*) + \frac{L}{2}\lVert {y}^{(\tau)}-x_{i}^{(\tau-1)} \rVert^2 \nonumber\\
		&\quad \quad \qquad \quad \,\,\,  +\langle \nabla f_i(x_{i}^{(\tau-1)}),{y}^{(\tau)}-x_{i}^{(\tau-1)} \rangle\Big) \nonumber \\
		& \leq \frac{1}{n}\sum_{i=1}^{n}  a_\tau\Big(\frac{L}{2}\lVert {y}^{(\tau)}-x_{i}^{(\tau-1)} \rVert^2 -\frac{\mu}{2}\lVert x_{i}^{(\tau-1)}-x^* \rVert^2  \nonumber\\
		&\quad \quad \qquad \quad \,\,\,+ \langle \nabla f_i(x_{i}^{(\tau-1)}),{y}^{(\tau)}-x^* \rangle \Big) \nonumber \\
		& = \frac{1}{n}\sum_{i=1}^{n}a_\tau\Big(\frac{L}{2}\lVert {y}^{(\tau)}-x_{i}^{(\tau-1)} \rVert^2 -\frac{\mu}{2}\lVert x_{i}^{(\tau-1)}-x^* \rVert^2\Big) \nonumber\\
		&\quad  + a_\tau\Big\langle \overline{g}^{(\tau-1)}, {y}^{(\tau)}-x^*\Big\rangle,
		\label{FDDA_smoothness_use}
	\end{align}
	where the two inequalities follow from~\eqref{eq:Lip-cond} and~\eqref{eq:strongly-convex}, respectively, and the equality uses the definition of $\overline{g}^{(\tau-1)}$.
	Upon summing up~\eqref{FDDA_smoothness_use} from $\tau = 1$ to $\tau = t$ and using Lemma~\ref{lem:descent} and $F = \frac{1}{n}\sum_{i=1}^n f_i + h$, we obtain
	\begin{align}
		& \sum_{\tau=1}^t a_{\tau} \left( F(y^{(\tau)}) - F(x^*) \right) \nonumber \\
		& \leq \frac{1}{n}\sum_{\tau=1}^t\sum_{i=1}^{n}a_\tau\Big(\frac{L}{2}\lVert {y}^{(\tau)}-x_{i}^{(\tau-1)} \rVert^2 -\frac{\mu}{2}\lVert x_{i}^{(\tau-1)}-x^* \rVert^2\Big) \nonumber\\ 
		& \quad + \frac{\mu}{2}\sum_{\tau=1}^ta_{\tau}\left(\|\overline{x}^{(\tau-1)} - x^*\|^2 - \|\overline{x}^{(\tau - 1)} - y^{(\tau)}\|^2\right) \nonumber \\
		&\quad -\frac{1}{2}\sum_{\tau=1}^t(1+\mu A_{\tau-1})\|y^{(\tau)} - y^{(\tau-1)}\|^2+ d(x^*) .  
	\end{align}
	Using the definition of $\Delta\mathbf{y}^{(\tau-1)}$ and the fact 
	\begin{align*}
		&\lVert \overline{x}^{(\tau-1)}-x^* \rVert^2  = \left\|\frac{1}{n}\sum_{i=1}^n x_i^{(\tau-1)} - x^*\right\|^2	\\ &\leq \frac{1}{n}\sum_{i=1}^{n}\lVert x_{i}^{(\tau-1)}-x^* \rVert^2,
	\end{align*}
	one can simplify the above inequality to
	\begin{align}
		&\sum_{\tau=1}^t a_{\tau} \left( F(y^{(\tau)}) - F(x^*) \right) \nonumber  \\
		&\leq \frac{1}{n}\sum_{\tau=1}^t\sum_{i=1}^n a_{\tau} \left(\frac{L}{2}\|y^{(\tau)} - x_i^{(\tau - 1)}\|^2 - \frac{\mu}{2}\|y^{(\tau)} - \overline{x}^{(\tau-1)}\|^2\right)\nonumber\\
		& \quad - \frac{1}{n}\sum_{\tau=1}^t\frac{1 + \mu A_{\tau-1}}{2}\|\Delta\mathbf{y}^{(\tau-1)}\|^2 + d(x^*). \label{eq:inter-step}
	\end{align}
	By the definition of $\overline{x}^{(\tau-1)}$, $\mathbf{x}^{(\tau)}$, and $\mathbf{y}^{(\tau)}$, one can verify
	\begin{equation*}
		\begin{split}
			&\sum_{i=1}^{n} \lVert y^{(\tau)}-\overline{x}^{(\tau-1)} \rVert^2 \\ &= \sum_{i=1}^{n} \left(\lVert \overline{x}^{(\tau-1)}\rVert^2 - \lVert x_{i}^{(\tau-1)}\rVert^2 + \lVert y^{(\tau)}-x_{i}^{(\tau-1)} \rVert^2\right) \\
			&= \sum_{i=1}^{n}\left( \lVert \overline{x}^{(\tau-1)}-y^{(\tau-1)}\rVert^2 - \lVert x_{i}^{(\tau-1)}-y^{(\tau-1)}\rVert^2 \right) \\
			&\quad  + \sum_{i=1}^n\|y^{(\tau)} - x_i^{(\tau-1)}\|^2 \\
			&\geq \sum_{i=1}^n \left( \|y^{(\tau)} - x_i^{(\tau-1)}\|^2 - \lVert x_{i}^{(\tau-1)}-y^{(\tau-1)}\rVert^2 \right) \\
			&= \|\mathbf{y}^{(\tau)} - \mathbf{x}^{(\tau-1)}\|^2 - \|\mathbf{y}^{(\tau-1)} - \mathbf{x}^{(\tau-1)}\|^2.
		\end{split}
	\end{equation*}
	Besides, recall that $F$ is convex , $\tilde{y}^{(t)} = A_t^{-1}\sum_{\tau=1}^t a_\tau y^{(\tau)}$, and $A_t = \sum_{\tau=1}^ta_\tau$.
	These, together with~\eqref{eq:inter-step}, yield
	\begin{small}
	\begin{align*}
		&A_t\left(F(\tilde{y}^{(t)})  -F(x^*)\right)  \\
		&\leq \frac{1}{n}\sum_{\tau=1}^ta_\tau\left(\frac{L-\mu}{2}\|\mathbf{y}^{(\tau)} - \mathbf{x}^{(\tau-1)}\|^2 + \frac{\mu}{2}\|\mathbf{x}^{(\tau-1)} - \mathbf{y}^{(\tau-1)}\|^2\right) \\
		& \quad  -\frac{1}{n}\sum_{\tau=1}^t\frac{1 + \mu A_{\tau-1}}{2}\|\Delta\mathbf{y}^{(\tau-1)}\|^2+ d(x^*) .
	\end{align*}
\end{small}
	Upon using the inequality
	$$ \|\mathbf{y}^{(\tau)} - \mathbf{x}^{(\tau-1)}\|^2 \leq 2\|\Delta\mathbf{y}^{(\tau-1)}\|^2 + 2\|\mathbf{y}^{(\tau-1)} - \mathbf{x}^{(\tau-1)}\|^2, $$
	we further obtain
	\begin{align*}
		& A_t\left(F(\tilde{y}^{(t)})  -F(x^*)\right) \\
		& \leq \frac{1}{n}\sum_{\tau=1}^t a_\tau\left(L-\mu - \frac{1 + \mu A_{\tau-1}}{2a_\tau}\right)\|\Delta\mathbf{y}^{(\tau-1)}\|^2\\
		&\quad+ \frac{2L - \mu}{2n}\sum_{\tau=1}^t a_\tau\|\mathbf{x}^{(\tau-1)} - \mathbf{y}^{(\tau-1)}\|^2 + d(x^*) \\
		& = \frac{2L - \mu - \frac{1}{a}}{2n}\sum_{\tau=1}^t a_\tau\|\Delta\mathbf{y}^{(\tau-1)}\|^2\\
		&\quad + \frac{2L - \mu}{2n}\sum_{\tau=1}^t a_\tau\|\mathbf{x}^{(\tau-1)} - \mathbf{y}^{(\tau-1)}\|^2 + d(x^*),
	\end{align*}
	where the equality follows from the identity
	$ \frac{1 + \mu A_{\tau-1}}{a_\tau} = \frac{1 - a\mu}{a}, $
	which holds due to the update rule of $\{a_t\}_{t\geq 0}$ and $\{A_t\}_{t\geq 0}$. Upon taking expectation on both sides of the above inequality and using Lemma~\ref{consensus_error_lemma}, one has
	\begin{align}
		&A_t\left(\mathbb{E}[F(\tilde{y}^{(t)})] - F(x^*)\right) + \frac{\gamma}{2n}\sum_{\tau=1}^ta_{\tau}\mathbb{E}[\|\Delta\mathbf{y}^{(\tau-1)}\|^2] \nonumber \\
		&\leq d(x^*) + \frac{(2L - \mu)a\sigma^2}{n\theta(L+\mu)^2} = C, \label{eq:last}
	\end{align}
	where $\gamma>0$ is defined in~\eqref{2ndcondition}. This implies~\eqref{convergence} as desired. Moreover, it follows from~\eqref{eq:last} and $A_t\left(\mathbb{E}[F(\tilde{y}^{(t)})]  - F(x^*)\right) \geq 0$ that
	$$ \sum_{\tau=1}^ta_{\tau}\mathbb{E}[\|\Delta\mathbf{y}^{(\tau-1)}\|^2] \leq \frac{2nC}{\gamma}. $$
	This, together with the convexity of $\lVert \cdot \rVert^2$, Jensen's Inequality, $a_t\leq a_{t+1}$ for all $t\geq 0$, and Lemma~\ref{consensus_error_lemma}, yields
	\begin{align*}
		&	{A_t}\mathbb{E}[\lVert{\tilde{\bf x}}^{(t)}-\tilde{\bf y}^{(t)}\rVert^2] \leq \sum_{\tau=1}^{t}{a_{\tau}}\mathbb{E}[\lVert{{\bf x}}^{(\tau)}-{{\bf y}}^{(\tau)}\rVert^2]  \\
		&\leq \sum_{\tau=0}^{t}{a_{\tau+1}}\mathbb{E}[\lVert{{\bf x}}^{(\tau)}-{{\bf y}}^{(\tau)}\rVert^2] \\
		& \leq \frac{2}{\eta}\sum_{\tau=0}^{t-1}{a_{\tau+1}}\mathbb{E}[\lVert \Delta{\bf y}^{(\tau)}\rVert^2] + \frac{2a\sigma^2}{\theta(L+\mu)^2}  \\
		&\leq \frac{4nC}{\eta\gamma} + \frac{2a\sigma^2}{\theta(L+\mu)^2} = D,
	\end{align*}
	which implies~\eqref{eq:x-y-dist} as desired.
\end{proof}

\section{Proofs of Corollaries~\ref{coro:1} and~\ref{coro:2}}
\label{appendix:pf-coro}
In this section, we provide the proofs of Corollary~\ref{coro:1} and Corollary~\ref{coro:2}. 

\subsection{Proof of Corollary~\ref{coro:1}}
\label{appendix:pf-coro1}
\begin{proof}[Proof of Corollary~\ref{coro:1}]
	Since $\mu>0$, we obtain from the update of $A_t$ in Algorithm~\ref{LinDDA} that
	$$ \frac{1}{A_t} = \frac{\mu}{\left(\frac{1}{1 - a\mu}\right)^t - 1} \leq \frac{(1 - a\mu)^t}{a}. $$
	Besides, upon using the fact that $F$ is strongly convex with modulus $\mu$, one has that for all $t\geq 0$ and $i=1,\dots,n$,
	\begin{align*}
		\|x_i^{(t)} - x^*\|^2 &\leq 2\|x_i^{(t)} - y^{(t)}\|^2 + 2\|y^{(t)} - x^*\|^2 \\
		& \leq 2\|x_i^{(t)} - y^{(t)}\|^2 + \frac{4}{\mu}\left( F(y^{(t)}) - F(x^*) \right).
	\end{align*}
	These, together with~\eqref{convergence} and~\eqref{eq:x-y-dist}, yields~\eqref{eq:linear-conv}.
\end{proof}

\subsection{Proof of Corollary~\ref{coro:2}}\label{sec:proof-corollary2}

\begin{proof}[Proof of Corollary~\ref{coro:2}]
	The upper bounds in~\eqref{eq:y-sublinear} and~\eqref{eq:x-to-y} directly follow from the results in Theorem~\ref{inexact_thm} and 
	\begin{equation*}
		\frac{1}{A_t} = \frac{1}{at}.
	\end{equation*}
	For the special case $h(x)=0$ and $d(x)=\frac{1}{2}\lVert x \rVert^2$, we consider
	\begin{equation}\label{unconstrained}
		\begin{aligned}
			& F(x_{i}^{(\tau)})-F(y^{(\tau)})=f(x_{i}^{(\tau)})-f(y^{(\tau)})\\
			& \leq \frac{1}{n}\sum_{j=1}^{n} \left( f_j(x_{i}^{(\tau)})-\langle \nabla f_j(x_{j}^{(\tau)}),y^{(\tau)}-x_{j}^{(\tau)} \rangle- f_j(x_{j}^{(\tau)}) \right)\\
			& \leq  \frac{1}{n}\sum_{j=1}^{n} \Big( \langle \nabla f_j(x_{j}^{(\tau)}),x_{i}^{(\tau)}-x_{j}^{(\tau)} \rangle-\langle \nabla f_j(x_{j}^{(\tau)}),y^{(\tau)}-x_{j}^{(\tau)} \rangle\\
			&\quad\quad\quad\quad \,\,+\frac{L}{2}\lVert x_{i}^{(\tau)}-y^{(\tau)}+y^{(\tau)}-x_{j}^{(\tau)}  \rVert^2 \Big) \\
			& =  \frac{1}{n}\sum_{j=1}^{n} \Big( \langle \nabla f_j(x_{j}^{(\tau)}),x_{i}^{(\tau)}-y^{(\tau)} \rangle+{L}\lVert x_{i}^{(\tau)}-y^{(\tau)}  \rVert^2\\
			&\quad\quad\quad\quad \,\, +{L}\lVert y^{(\tau)}-x_{j}^{(\tau)}  \rVert^2 \Big) \\
			& =  \left\langle \overline{g}^{(\tau)},x_{i}^{(\tau)}-y^{(\tau)} \right\rangle+{L}\lVert x_{i}^{(\tau)}-y^{(\tau)}  \rVert^2 +\frac{L}{n}\lVert \mathbf{y}^{(\tau)}-\mathbf{x}^{(\tau)}  \rVert^2,
		\end{aligned}
	\end{equation}
	where the two inequalities follow from~\eqref{eq:strongly-convex} and~\eqref{eq:Lip-cond}, respectively.
	The closed-form solutions for~\eqref{primal} and~\eqref{eq:def-y} can be derived as
	\begin{equation*}
		x_{i}^{(\tau)}=-\frac{z_{i}^{(\tau)}}{1+\mu A_\tau}, \quad
		y^{(\tau)}=-\frac{\overline{z}^{(\tau)}}{1+\mu A_\tau}.
	\end{equation*}
	Therefore $y^{(\tau)}=\overline{x}^{(\tau)}$.
	We sum up~\eqref{unconstrained} from $i= 1$ to $i = n$ to get
	\begin{equation}\label{convergence:xtoy}
		\begin{split}
			\sum_{i=1}^{n}\left(F(x_{i}^{(\tau)})-F(y^{(\tau)})\right) 
			\leq 2L\lVert \mathbf{y}^{(\tau)}-\mathbf{x}^{(\tau)}  \rVert^2.
		\end{split}
	\end{equation}
	Upon summing up~\eqref{convergence:xtoy} from $\tau = 1$ to $\tau = t$ and using the convexity of $F$, we obtain
	\begin{equation}
		\begin{split}
			&t\sum_{i=1}^{n}\left(F(\tilde{x}_{i}^{(t)})-F(y^{(\tau)})\right) \leq  \sum_{\tau=1}^{t} \sum_{i=1}^{n}\left(F(x_{i}^{(\tau)})-F(y^{(\tau)})\right) \\
			&\leq 2L\sum_{\tau=1}^{t}\lVert \mathbf{y}^{(\tau)}-\mathbf{x}^{(\tau)}  \rVert^2,
		\end{split}
	\end{equation}
	where $\tilde{x}_{i}^{(t)}= \frac{1}{t}\sum_{\tau=1}^{t}x_i^{(\tau)}$.
	After taking expectation on both sides of the above inequality and
	using Lemma~\ref{consensus_error_lemma} with $a_{\tau}=a$, we get
	\begin{equation}\label{convegence_x}
		\begin{split}
			&t\sum_{i=1}^{n}\mathbb{E}\left[ F(\tilde{x}_{i}^{(t)})-F(y^{(t)})\right]\\
			&\leq  \frac{4L}{\eta}\sum_{\tau=1}^{t}\mathbb{E}[\lVert \Delta{\bf y}^{(\tau-1)}\rVert^2]+  \frac{4L\sigma^2}{\theta (L+\mu)^2}\\
			&=\frac{4L}{(1-\nu)^2}\sum_{\tau=1}^{t}\mathbb{E}[\lVert \Delta{\bf y}^{(\tau-1)}\rVert^2]+  \frac{4\sigma^2}{ L(1-\nu^2)}.
		\end{split}
	\end{equation}
	By setting $\mu=0$ and $d(x^*)=\frac{1}{2}\lVert x^* \rVert^2$ in~\eqref{eq:last}, we have
	\begin{equation}
		\begin{split}
		&a t\left(\mathbb{E}[F(\tilde{y}^{(t)})] - F(x^*)\right) \\
		&  \leq -\left(\frac{1}{2a}-L-\frac{2L}{(1-\nu^2)}\right) \frac{a}{n}\sum_{\tau=1}^t \mathbb{E}[\|\Delta\mathbf{y}^{(\tau-1)}\|^2] \\
		&\quad +\frac{\lVert x^*\rVert^2}{2} + \frac{2a\sigma^2}{nL(1-\nu^2)}.	
	\end{split}
	\end{equation}
	Also, by multiplying $n/a>0$ on both sides of the above inequality and adding the resultant inequality to~\eqref{convegence_x}, we obtain
	\begin{equation}
		\begin{split}
			&	t\left(\mathbb{E}[F(\tilde{x}_{i}^{(t)})]-F(x^*)\right)\leq 
			t\sum_{i=1}^{n}\left(\mathbb{E}[F(\tilde{x}_{i}^{(t)})]-F(x^*)\right) \\
			& \leq
			-\Big( \frac{1}{2a}-{L}    -\frac{6L}{(1-\nu)^2}\Big){\sum_{\tau=1}^{t}  \mathbb{E}[  \lVert \Delta\mathbf{y}^{(\tau-1)} \rVert^2]}\\
			& \quad +\frac{n}{2a}\lVert x^* \rVert^2+ \frac{6\sigma^2}{L\big(1-\nu^2\big)}.
		\end{split}
	\end{equation}
	Now, using the condition in~\eqref{x_condition_smooth}, we arrive at~\eqref{eq:x-h=0}.
\end{proof}

\section{Proofs of Supporting Lemmas for Theorem~\ref{inexact_thm}}
\label{supporting_lemmas}
\subsection{Proof of Lemma~\ref{lem:M-radius}}\label{sec:proof-Lemma1}

\begin{proof}[Proof of Lemma~\ref{lem:M-radius}]
	We first show that $\nu$ monotonically increases with $a$ if $a\in(0, 1/\mu)$. Recall that $\mathbf{M}$ is defined in~\eqref{eq:matrix-M}. Then, the characteristic polynomial of $\mathbf{M}$, denoted by $p(\lambda)$, is a quadratic function:
	\begin{equation}
		\begin{split}
		\label{eq:cha-poly}
		p(\lambda) :=& \det(\lambda I - \mathbf{M}) = (\lambda - M_{11})(\lambda - M_{22}) - M_{12}M_{21} \\
		=& \lambda^2 - \frac{\beta(2 + aL)}{1 - a\mu}\lambda + \frac{\beta^2}{1 - a\mu} - \frac{a\beta(L + \mu)}{(1 - a\mu)^2}.
	\end{split}
	\end{equation}
	Using this, we obtain that $\mathbf{M}$ has two real eigenvalues $\lambda_1 =(\xi_1+\xi_2)/2$ and $ \lambda_2 =(\xi_1-\xi_2)/2$,
	where 
	\begin{equation}
		\label{def:xi}
		\begin{split}
			\xi_1 = \frac{\beta(2+aL)}{1-a\mu}, \qquad \xi_2 = \frac{\sqrt{{a^2\beta^2L^2}+{4a\beta(\beta+1)}{(L+\mu)}}}{1-a\mu}.
		\end{split}
	\end{equation}
	Notice that $\xi_1 > 0$ and $\xi_2>0$ for any $a\in(0,1/\mu)$. Thus, we have $\lambda_1 > 0$ and $|\lambda_1| > |\lambda_2|$ for any $a\in(0,1/\mu)$. It then follows that $\rho(\mathbf{M}) = \lambda_1$ and
	\begin{equation}
		\begin{split}\label{nu_of_a}
		&\nu(a)=\rho(\mathbf{M})\sqrt{1 - a\mu} = \lambda_1\sqrt{1 - a\mu}\\
		&=\frac{\beta(2+aL)}{2\sqrt{1-a\mu}}+\frac{\sqrt{{a^2\beta^2L^2}+{4a\beta(\beta+1)}{(L+\mu)}}}{2\sqrt{1-a\mu}}.
		\end{split}
	\end{equation}
	By routine calculation, one can verify that $\nu'(a)>0$ if $a\in(0, {\mu}^{-1})$. Therefore, the value of $\nu$ monotonically increases with $a$ if $a\in(0, {\mu}^{-1})$.
	
	Next, we show that $\nu < 1$ if~\eqref{eq:condition1-a} is satisfied. Note that~\eqref{eq:condition1-a} implies that $0 < a < \mu^{-1}$ and 
	$
			{\beta (2L+3\mu)}/{(1-\beta)^2}<{1}/{a}-\mu={(1-a\mu)}/{a}.
	$
	It then follows that $1 - a\mu \in (0,1]$ and hence
	\begin{equation*}
		\begin{split}
			0&< (1-\beta)^2-\frac{a\beta(2L+3\mu)}{1-a\mu}\\
			&=	1+{\beta^2}-\left(\beta + \frac{\beta+a\beta(L+\mu)}{1-a\mu}\right)-\frac{a\beta(L+\mu)}{1-a\mu}.
		\end{split}
	\end{equation*}
	Upon dividing both sides of the above inequality by $1-a\mu$, we obtain
	\begin{align}
		0 & <	\frac{1}{1-a\mu}+\frac{\beta^2}{1-a\mu}-\frac{1}{1-a\mu}\big(\beta + \frac{\beta+a\beta(L+\mu)}{1-a\mu}\big)\nonumber\\
		&\quad -\frac{a\beta(L+\mu)}{(1-a\mu)^2} \nonumber \\
		& \leq \frac{1}{1-a\mu}+\frac{\beta^2}{1-a\mu}-\frac{1}{\sqrt{1-a\mu}}\big(\beta + \frac{\beta+a\beta(L+\mu)}{1-a\mu}\big)\nonumber\\
		&\quad -\frac{a\beta(L+\mu)}{(1-a\mu)^2} \nonumber \\
		& = p\left(\frac{1}{\sqrt{1 - a\mu}}\right), \label{eq:value-p}
	\end{align}
	where the second inequality is due to $1 - a\mu \in (0,1]$ and the equality follows from~\eqref{eq:cha-poly}. Besides, using the definition of characteristic polynomial, one further has
	\begin{align}
		 0 &< p\left(\frac{1}{\sqrt{1 - a\mu}}\right)\nonumber \\
	& = \left( \frac{1}{\sqrt{1 - a\mu}} - M_{11} \right)\left( \frac{1}{\sqrt{1 - a\mu}} - M_{22}\right) - M_{12}M_{21}. 
	\end{align}
	By~\eqref{eq:matrix-M}, $\beta\in(0,1)$, and $1 - a\mu\in(0,1]$, we have that $M_{12} > 0$, $M_{21} > 0$, and $1/\sqrt{1 - a\mu} > M_{11}$. It then follows that $1/\sqrt{1 - a\mu} > M_{22}$ and hence
	$$ p^\prime\left(\frac{1}{\sqrt{1 - a\mu}}\right) = \frac{2}{\sqrt{1 - a\mu}} - M_{11} - M_{22} > 0. $$
	This, together with the fact that $q$ is a quadratic function, implies that $q(\lambda)$ is monotonically increasing on $[1/\sqrt{1 - a\mu}, \infty)$. It then follows from~\eqref{eq:value-p} that 
	$ \frac{1}{\sqrt{1 - a\mu}} > \lambda_1 = \rho(\mathbf{M}), $
	which implies that $\nu < 1$. 
\end{proof}

\subsection{Proof of Lemma~\ref{consensus_error_lemma}}\label{sec:proof-Lemma2}
In this subsection, we first present three technical lemmas, and then provide the proof of Lemma~\ref{consensus_error_lemma}. 

\begin{lemma} \label{conservation_property}
	For the sequences $\{\overline{s}^{(t)}\}_{t\geq 0}$ and $\{\overline{z}^{(t)}\}_{t\geq 0}$ defined in~\eqref{eq:notation-2}, one has that for any $t\geq 0$,
	\begin{equation}\label{eq:overline-sz}
		\begin{split}
			\overline{s}^{(t)} = \overline{g}^{(t)}-\mu\overline{x}^{(t)}, \qquad
			\overline{z}^{(t)} = \sum_{\tau=0}^{t-1}a_{\tau+1}\overline{s}^{(\tau)}.
		\end{split}
	\end{equation}	
\end{lemma} 
\begin{proof}
	We prove by an induction argument. Since $s_i^{(0)} = \nabla f_i(x^{(0)}) - \mu x^{(0)}$, $z_i^{(0)} = 0$ and $x_i^{(0)} = x^{(0)}$ for all $i$, we readily have that~\eqref{eq:overline-sz} holds when $t = 0$. Now, suppose that~\eqref{eq:overline-sz} holds for $t-1$. From~\eqref{eq:notation-1} and~\eqref{eq:notation-2}, we observe that the following identities hold for any $\tau \geq 0$:
	\begin{equation}\label{eq:kro-form}
		\begin{aligned}
		&\overline{x}^{(\tau)} = \frac{1}{n}(\mathbf{1}^T\otimes I)\mathbf{x}^{(\tau)}, \quad 
		\overline{g}^{(\tau)} = \frac{1}{n}(\mathbf{1}^T\otimes I)\nabla^{(\tau)}, \\
		&
		\overline{s}^{(\tau)} = \frac{1}{n}(\mathbf{1}^T\otimes I)\mathbf{s}^{(\tau)}, \quad
		\overline{z}^{(\tau)} = \frac{1}{n}(\mathbf{1}^T\otimes I)\mathbf{z}^{(\tau)}.
		\end{aligned}
	\end{equation}
	It then follows from this and~\eqref{consensus_on_s} that
	\begin{align*}
		&\overline{s}^{(t)}  = \frac{1}{n}(\mathbf{1}^T\otimes I)\mathbf{s}^{(t)}\\
		&= \frac{1}{n}(\mathbf{1}^T\otimes I)(P^{(t-1)}\otimes I){\bf s}^{(t-1)} + \frac{1}{n}(\mathbf{1}^T\otimes I)\nabla^{(t)}\\
		&\quad-\frac{1}{n}(\mathbf{1}^T\otimes I)\nabla^{(t-1)} - \frac{\mu}{n}(\mathbf{1}^T\otimes I)\Delta{ \bf x}^{(t-1)} \\
		&= \frac{1}{n}(\mathbf{1}^TP^{(t-1)}\otimes I){\bf s}^{(t-1)} + \overline{g}^{(t)} - \overline{g}^{(t-1)} - \mu\overline{x}^{(t)} + \mu\overline{x}^{(t-1)} \\
		&= \frac{1}{n}(\mathbf{1}^T\otimes I){\bf s}^{(t-1)} + \overline{g}^{(t)} - \overline{g}^{(t-1)} - \mu\overline{x}^{(t)} + \mu\overline{x}^{(t-1)} \\
		&= \overline{s}^{(t-1)} + \overline{g}^{(t)} - \overline{g}^{(t-1)} - \mu\overline{x}^{(t)} + \mu\overline{x}^{(t-1)} \\
		&= \overline{g}^{(t)} - \mu\overline{x}^{(t)},
	\end{align*}
	where the second equality is due to~\eqref{consensus_on_s}, the third equality uses the fact that $(A\otimes B)(C\otimes D) = (AC \otimes BD)$, the fourth equality follows from the fact that $P^{(t-1)}$ is doubly stochastic, and the last equality is due to the assumption that~\eqref{eq:overline-sz} holds for $t-1$. Similarly, by~\eqref{concise_2nd_consensus} and~\eqref{eq:kro-form}, we obtain
	\begin{align*}
		&\overline{z}^{(t)}  = \frac{1}{n}(\mathbf{1}^T\otimes I)\mathbf{z}^{(t)} \\
		&= \frac{1}{n}(\mathbf{1}^T\otimes I)(P^{(t-1)}\otimes I)\left({\bf z}^{(t-1)} + a_t\mathbf{s}^{(t-1)}\right)\\
		& = \frac{1}{n}(\mathbf{1}^T\otimes I)\left({\bf z}^{(t-1)} + a_t\mathbf{s}^{(t-1)}\right) \\
		& = \overline{z}^{(t-1)} + a_t\overline{s}^{(t-1)} 
		= \sum_{\tau = 0}^{t-1} a_{\tau+1}\overline{s}^{(\tau)}.
	\end{align*}
	Therefore,~\eqref{eq:overline-sz} holds for $t$ and the induction argument is completed. 
\end{proof}

\begin{lemma} \label{gamma_continuity}
	For the sequence $\{x_i^{(t)}: i = 1,\dots, n\}_{t\geq 0}$ generated by Algorithm~\ref{LinDDA} and the auxiliary sequence $\{y^{(t)}\}_{t\geq 0}$ defined in~\eqref{eq:def-y}, one has that for all $t\geq 0$ and $i = 1,\dots,n$,
	\begin{equation} \label{XYrelation}
		\begin{split}
			\lVert x_{i}^{(t)} -y^{(t)} \rVert & \leq  \frac{1}{1+\mu A_t} \big\lVert z_{i}^{(t)}-\overline{z}^{(t)}\rVert,
		\end{split}
	\end{equation}
	where $\overline{z}^{(t)}$ is defined in~\eqref{eq:notation-2}.
	
\end{lemma}
\begin{proof}
	It is easy to see that~\eqref{XYrelation} holds when $t = 0$ because both sides of~\eqref{XYrelation} equal $0$. Now, suppose that $t\geq 1$. Recall that $d$ is strongly convex with modulus $1$. Let the mapping $R:\mathbb{R}^m\rightarrow\mathbb{R}^m$ be defined as
	$$ R(\omega) := \argmin_{x\in\mathbb{R}^m} \left\{\langle \omega, x\rangle + \phi(x)\right\}, $$
	where $\phi(x) = A_t(\mu\|x\|^2/2 + h(x)) + d(x)$ is strongly convex with modulus $1 + \mu A_t$. Then, by~\eqref{primal} and~\eqref{eq:def-y}, we have
	$$ y^{(t)} = R(\overline{z}^{(t)}), \qquad x_i^{(t)} = R(z_i^{(t)}), \quad \forall i = 1,\dots,n. $$
	Moreover, the mapping $R$ is Lipschitz continuous with Lipschitz constant $(1 + \mu A_t)^{-1}$; see, e.g.,~\cite[Proposition 4.9]{juditsky2019unifying}. This immediately implies~\eqref{XYrelation} as desired.
\end{proof}



Next, we recall a lemma from~\cite[Lemma 4]{xu2017convergence}. 
\begin{lemma}
	\label{pertubed_consensus}
	Suppose that $\{q^{(t)}\}_{t\geq 0}$ and $\{p^{(t)}\}_{t\geq 0}$ are two sequences of positive scalars such that for all $t\geq 0$, 
	\begin{equation*}
		q^{(t)}\leq \nu^t c+\sum_{\tau=0}^{t-1}\nu^{t-\tau-1}p^{(\tau)}
	\end{equation*}
	where $\nu\in(0,1)$ and $c\geq 0$ is a constant. Then, the following holds for all $t\geq 0$:
	\begin{equation*}
		\sum_{\tau=1}^{t}(q^{(\tau)})^2\leq  \frac{2}{(1-\nu)^2}\sum_{\tau=0}^{t-1}(p^{(\tau)})^2+\frac{2c^2}{1-\nu^2}.
	\end{equation*}
\end{lemma}

Now we are ready to prove Lemma~\ref{consensus_error_lemma}.
\begin{proof}[Proof of Lemma~\ref{consensus_error_lemma}]
	From Lemma~\ref{conservation_property}, we have
	\begin{equation*}
		\overline{z}^{(\tau)} = \overline{z}^{(\tau-1)}+a_\tau \overline{s}^{(\tau-1)}.
	\end{equation*}
	This, together with~\eqref{concise_2nd_consensus} and the definition of $\tilde{\mathbf{z}}^{(\tau)}$ in~\eqref{eq:notation-3}, yields
	\begin{equation}\label{deviation_z}
		\begin{split}
			\tilde{\bf z}^{(\tau)} 
			&=  \mathbf{P}^{(\tau-1)}{\bf z}^{(\tau-1)}-\mathbf{1}\otimes\overline{z}^{(\tau-1)}\\
			&\quad +{{a_{\tau}}}\left(\mathbf{P}^{(\tau-1)}{\bf s}^{(\tau-1)} - \mathbf{1}\otimes\overline{s}^{(\tau-1)}\right).
		\end{split}
	\end{equation}
	It then follows from~\eqref{Minkowski} that
	\begin{equation}\label{eq:take-exp}
		\begin{split}
			\lVert	\tilde{\bf z}^{(\tau)} \rVert_\mathbb{E} 
			&\leq  \left\| \mathbf{P}^{(\tau-1)}{\bf z}^{(\tau-1)} - \mathbf{1}\otimes\overline{z}^{(\tau-1)}\right\|_\mathbb{E}\\
			&\quad+{{a_{\tau}}}\left\| \mathbf{P}^{(\tau-1)}{\bf s}^{(\tau-1)}-\mathbf{1}\otimes\overline{s}^{(\tau-1)}\ \right\|_\mathbb{E}.
		\end{split}
	\end{equation}
	Note that $\mathbf{1}\otimes\bar{z}^{(\tau-1)} = (\mathbf{1}\otimes I)\bar{z}^{(\tau-1)}$, which, together with~\eqref{eq:kro-form} and the identity $(A\otimes B)(C\otimes D) = (AC\otimes BD)$, yields
	$$ \mathbf{1}\otimes\bar{z}^{(\tau-1)} = \frac{1}{n}(\mathbf{1}\otimes I)(\mathbf{1}^T\otimes I)\mathbf{z}^{(\tau-1)} = \left(\frac{\mathbf{1}\mathbf{1}^T}{n}\otimes I\right)\mathbf{z}^{(\tau-1)}. $$
	Using this and $\mathbf{P}^{(\tau-1)} = P^{(\tau-1)}\otimes I$, we obtain
	\begin{align*}
		&\mathbf{P}^{(\tau-1)}{\bf z}^{(\tau-1)}-\mathbf{1}\otimes\overline{z}^{(\tau-1)} \\
		& = (P^{(\tau-1)}\otimes I)\mathbf{z}^{(\tau-1)} - \left(\frac{\mathbf{1}\mathbf{1}^T}{n}\otimes I\right)\mathbf{z}^{(\tau-1)} \\
		& = \left(\left(P^{(\tau-1)}-\frac{\mathbf{1}\mathbf{1}^{T}}{n}\right)\otimes I\right){\bf z}^{(\tau-1)} \\
		& = \left(\left(P^{(\tau-1)}-\frac{\mathbf{1}\mathbf{1}^{T}}{n}\right)\otimes I\right)\left(\tilde{\bf z}^{(\tau-1)} + (\mathbf{1}\otimes I)\bar{z}^{(\tau-1)}\right) \\
		& = \left(\left(P^{(\tau-1)}-\frac{\mathbf{1}\mathbf{1}^{T}}{n}\right)\otimes I\right)\tilde{\bf z}^{(\tau-1)} \\
		&\quad+ \left(\left(P^{(\tau-1)}\mathbf{1}-\mathbf{1}\right)\otimes I\right)\bar{z}^{(\tau-1)} \\
		& = \left(\left(P^{(\tau-1)}-\frac{\mathbf{1}\mathbf{1}^{T}}{n}\right)\otimes I\right)\tilde{\bf z}^{(\tau-1)},
	\end{align*}
	where the third equality uses~\eqref{eq:notation-3} and $\mathbf{1}\otimes\bar{z}^{(\tau-1)} = (\mathbf{1}\otimes I)\bar{z}^{(\tau-1)}$, the fourth equality follows from the identity $(A\otimes B)(C\otimes D) = (AC\otimes BD)$, and the last one is due to the fact that $P^{(\tau-1)}$ is doubly stochastic. Then, by~\eqref{expectation_operator} and Assumption~\ref{graphconnected}, one has
	\begin{small}
	\begin{equation*}
		\begin{split}
			&\left\| \mathbf{P}^{(\tau-1)}{\bf z}^{(\tau-1)}-(\mathbf{1}\otimes I)\overline{z}^{(\tau-1)} \right\|_{\mathbb{E}}\\
			& =\left\| \left(\left(P^{(\tau-1)}-\frac{\mathbf{1}\mathbf{1}^{T}}{n}\right)\otimes I\right)\tilde{\bf z}^{(\tau-1)} \right\|_{\mathbb{E}}\\
			& \overset{(i)}{\leq } \lVert	\tilde{\bf z}^{(\tau-1)} \rVert_\mathbb{E}\sqrt{ \rho \left(\mathbb{E}\left[\left(P^{(\tau-1)}-\frac{\mathbf{1}\mathbf{1}^{T}}{n}\right)^{T}\left(P^{(\tau-1)}-\frac{\mathbf{1}\mathbf{1}^{T}}{n}\right)  \right] \right) }	  \\
			& = \lVert	\tilde{\bf z}^{(\tau-1)} \rVert_\mathbb{E}  \sqrt{ \rho \left(\mathbb{E}\left[P^{{(\tau-1)}^{T}}P^{(\tau-1)}\right] -\frac{\mathbf{1}\mathbf{1}^{{T}}}{n}  \right) }	 \\
			& \leq \beta\lVert	\tilde{\bf z}^{(\tau-1)} \rVert_\mathbb{E}
		\end{split}
	\end{equation*}
\end{small}
	where 
	we use $(A\otimes B)^T = A^T\otimes B^T$ and $(A\otimes B)(C\otimes D) = AC\otimes BD$ and that $P^{(\tau-1)}$ is independent of $\tilde{\bf z}^{(\tau-1)}$ in (i). Using the same arguments as above, we have
	\begin{equation}\label{eq:P-s}
		\left\| \mathbf{P}^{(\tau-1)}{\bf s}^{(\tau-1)}-\mathbf{1}\otimes\overline{s}^{(\tau-1)}\ \right\|_\mathbb{E} \leq \beta\lVert	\tilde{\bf s}^{(\tau-1)} \rVert_\mathbb{E}. 
	\end{equation}
	It then follows from~\eqref{eq:take-exp} that
	\begin{equation}\label{dual_error}
		\begin{split}
			\lVert	\tilde{\bf z}^{(\tau)} \rVert_\mathbb{E} 
			\leq   \beta \left(\lVert	\tilde{\bf z}^{(\tau-1)} \rVert_\mathbb{E} 
			+a_{\tau}	\lVert	\tilde{\bf s}^{(\tau-1)} \rVert_\mathbb{E} \right).
		\end{split}
	\end{equation}
	Similarly, from Lemma~\ref{conservation_property},~\eqref{eq:notation-3}, and~\eqref{consensus_on_s}, we obtain
	\begin{equation}\label{eq:tilde-s}
		\begin{split}
			\lVert	\tilde{\bf s}^{(\tau)} \rVert_\mathbb{E} 
			&  \leq  \left\| \mathbf{P}^{(\tau-1)}{{\bf s}}^{(\tau-1)}-\left(\mathbf{1}\otimes I\right)\overline{s}^{(\tau-1)}\right\|_{\mathbb{E}}\\
			&\quad +\left\| \nabla^{(\tau)}- \nabla^{(\tau-1)}-\left(\mathbf{1}\otimes I\right)(\overline{g}^{(\tau)}- \overline{g}^{(\tau-1)})\right\|_\mathbb{E}\\
			& \quad +\mu\left\|(\mathbf{1}\otimes I)\Delta{ \overline{x}}^{(\tau-1)} -\Delta{ \bf x}^{(\tau-1)}\right\|_\mathbb{E}.
		\end{split}
	\end{equation}
	By~\eqref{eq:kro-form}, one can verify that
	\begin{align*}
		&\nabla^{(\tau)}- \nabla^{(\tau-1)}-(\mathbf{1}\otimes I)(\overline{g}^{(\tau)}-\overline{g}^{(\tau-1)})\\
		 &\quad \quad \quad \qquad \quad \,\,\,\, \quad = \left( \left(I-\frac{\mathbf{1}\mathbf{1}^{{T}}}{n}\right) \otimes I \right) \left(\nabla^{(\tau)}- \nabla^{(\tau-1)}\right), \\
		&\Delta{ \bf x}^{(\tau-1)} - \left(\mathbf{1}\otimes I\right)\Delta{\overline{x}}^{(\tau-1)}  = \left( \left(I-\frac{\mathbf{1}\mathbf{1}^{{T}}}{n}\right) \otimes I \right)\Delta{ \bf x}^{(\tau-1)},
	\end{align*} 
	which respectively imply that
	\begin{small}
	\begin{equation*}
		\begin{split}
			&\lVert \nabla^{(\tau)}- \nabla^{(\tau-1)}-(\mathbf{1}\otimes I)(\overline{g}^{(\tau)}-\overline{g}^{(\tau-1)})\rVert_\mathbb{E}\\
			& \leq   \left\| \left(I-\frac{\mathbf{1}\mathbf{1}^{{T}}}{n}\right) \otimes I \right\| \lVert \nabla^{(\tau)}- \nabla^{(\tau-1)} \rVert_\mathbb{E} \leq \lVert \nabla^{(\tau)}- \nabla^{(\tau-1)} \rVert_\mathbb{E}, \\
			&\lVert \Delta{ \bf x}^{(k-1)}- \left(\mathbf{1}\otimes I\right)\Delta{\overline{x}}^{(k-1)}\rVert_\mathbb{E}\\
			& \leq  \left\| \left(I-\frac{\mathbf{1}\mathbf{1}^{{T}}}{n}\right) \otimes I \right\| \lVert  \Delta \mathbf{x}^{(k-1)} \rVert_\mathbb{E} \leq  \lVert  \Delta \mathbf{x}^{(k-1)} \rVert_\mathbb{E}. 
		\end{split}
	\end{equation*} 
\end{small}
	Besides, it follows from~\eqref{eq:Lip-origin} that $\lVert \nabla^{(\tau)}- \nabla^{(\tau-1)} \rVert_\mathbb{E} \leq L\lVert  \Delta \mathbf{x}^{(k-1)} \rVert_\mathbb{E}$. Upon substituting these and~\eqref{eq:P-s} into~\eqref{eq:tilde-s}, we obtain
	\begin{equation}
		\label{iterations-s}
		\begin{split}
			\lVert	\tilde{\bf s}^{(\tau)} \rVert_\mathbb{E} 
			\leq \beta 	\lVert	\tilde{\bf s}^{(\tau-1)} \rVert_\mathbb{E}  + (L+\mu)\lVert \Delta{ \bf x}^{(\tau-1)} \rVert_\mathbb{E}.
		\end{split}
	\end{equation}
	Multiplying the both sides of the above inequality by $a_{\tau+1} = a_\tau/(1 - a\mu)$, we have
	\begin{equation}
		\begin{split}
			&a_{\tau+1} 	\lVert	\tilde{\bf s}^{(\tau)} \rVert _\mathbb{E}\\
			&\leq  {\frac{1}{1-a\mu}} \left(\beta a_{\tau}	\lVert	\tilde{\bf s}^{(\tau-1)} \rVert_\mathbb{E}  + (L+\mu){{a_{\tau}}} \lVert \Delta{ \bf x}^{(\tau-1)} \rVert_\mathbb{E} \right).
		\end{split}
	\end{equation}
	Upon using Lemma~\ref{gamma_continuity} and
	\begin{equation*}
		\begin{split}
			&\lVert \Delta{ \bf x}^{(\tau-1)} \rVert_\mathbb{E}\\
			&\leq \lVert {\bf x}^{(\tau)}-{{\bf y}}^{(\tau)}\rVert_\mathbb{E}+\lVert {\bf x}^{(\tau-1)}-{{\bf y}}^{(\tau-1)}\rVert_\mathbb{E} + \lVert \Delta{\bf y}^{(\tau-1)}\rVert_\mathbb{E},
		\end{split}
	\end{equation*}
	one has
	\begin{small}
	\begin{equation*}
		\begin{split}
			&{a_\tau} \lVert \Delta{ \bf x}^{(\tau-1)} \rVert_\mathbb{E}\\
			& \leq \frac{{a_\tau}}{1+\mu A_\tau} \lVert\tilde{\mathbf z}^{(\tau)}\rVert_\mathbb{E}+ \frac{{a_{\tau}}}{1+\mu A_{\tau-1}}\lVert\tilde{\mathbf z}^{(\tau-1)}\rVert_\mathbb{E}+ {a_\tau}\lVert \Delta{\bf y}^{(\tau-1)}\rVert_\mathbb{E} \\
			& = a \lVert\tilde{\mathbf z}^{(\tau)}\rVert_\mathbb{E}+ \frac{{a}}{1-a\mu}\lVert\tilde{\mathbf z}^{(\tau-1)}\rVert_\mathbb{E}+ {a_\tau}\lVert \Delta{\bf y}^{(\tau-1)}\rVert_\mathbb{E} ,
		\end{split}
	\end{equation*}
\end{small} 
where the equality follows from
 	\begin{equation*}
 		\frac{1+\mu A_\tau}{a_\tau}= \frac{(\frac{1}{1-a\mu})^\tau}{\frac{a}{1-a\mu}\Big( \frac{1}{1-a\mu}\Big)^{\tau-1}}=\frac{1}{a}.
 	\end{equation*}
In light of~\eqref{dual_error}, we have
	\begin{equation*}
		\begin{split}
			{a_\tau} \lVert \Delta{ \bf x}^{(\tau-1)} \rVert_\mathbb{E}
			&\leq  \left( a\beta +\frac{a}{1-a\mu}\right)\lVert\tilde{\mathbf z}^{(\tau-1)}\rVert_\mathbb{E} \\
			&\quad+ a\beta a_\tau\lVert\tilde{\mathbf s}^{(\tau-1)}\rVert_\mathbb{E} +{a_\tau}\lVert \Delta{\bf y}^{(\tau-1)}\rVert_\mathbb{E}.
		\end{split}
	\end{equation*}
Therefore
	\begin{equation}\label{gradient_error}
		\begin{split}
			a_{\tau+1} 	\lVert	\tilde{\bf s}^{(\tau)} \rVert_\mathbb{E} 	\leq& 
			\frac{  \beta +{a\beta(L+\mu)}}{1-a\mu}	a_{\tau} 	\lVert	\tilde{\bf s}^{(\tau-1)} \rVert_\mathbb{E}\\ 
			&+\frac{a(L+\mu)}{1-a\mu}\left( \beta + \frac{1}{1-a\mu}\right)	 	\lVert	\tilde{\bf z}^{(\tau-1)} \rVert_\mathbb{E}\\
			&	+\frac{L+\mu}{1-a\mu}a_\tau \lVert\Delta{\bf y}^{(\tau-1)}\rVert_\mathbb{E}.
		\end{split}
	\end{equation}
	By combining~\eqref{dual_error} and~\eqref{gradient_error}, the following inequality can be established:
	\begin{equation*}
		\begin{split}
			&\begin{bmatrix}
			\lVert	\tilde{\bf z}^{(\tau)} \rVert_\mathbb{E}  \\
				a_{\tau+1} 	\lVert	\tilde{\bf s}^{(\tau)} \rVert_\mathbb{E} 
			\end{bmatrix} \\
			&\leq  \mathbf{M} 	\begin{bmatrix}
				\lVert	\tilde{\bf z}^{(\tau-1)} \rVert_\mathbb{E}  \\
				a_{\tau} 	\lVert	\tilde{\bf s}^{(\tau-1)} \rVert_\mathbb{E} 
			\end{bmatrix}   + \frac{L+\mu}{{1-a\mu}} \begin{bmatrix}
				0\\
				{a_\tau}\lVert \Delta{\bf y}^{(\tau-1)}\rVert_\mathbb{E} 
			\end{bmatrix} 
		\end{split}
	\end{equation*}
	where $\mathbf{M}$ is defined in~\eqref{eq:matrix-M}.
	By iterating the preceding linear system inequality and using
	\begin{equation*}
		\lVert	\tilde{\bf z}^{(0)} 	\lVert= 0 ,\quad 	\lVert	\tilde{\bf s}^{(0)}	\lVert = \sqrt{\sum_{i=1}^n\left\|\nabla f_i(x^{(0)}) - \overline{g}^{(0)}\right\|^2} :=\sigma,
	\end{equation*}
	we obtain
	\begin{equation*}
		\begin{split}
			&\begin{bmatrix}
				\lVert	\tilde{\bf z}^{(t)} \rVert_\mathbb{E}  \\
				a_{t+1} 	\lVert	\tilde{\bf s}^{(t)} \rVert_\mathbb{E} 
			\end{bmatrix}\\
			\leq&\frac{L+\mu}{1-a\mu} \sum_{\tau=0}^{t-1}\mathbf{M}^{t-\tau-1} \sqrt{a_{\tau+1}}\begin{bmatrix}
				0\\
				\sqrt{a_{\tau+1}}\lVert \Delta{\bf y}^{(\tau)}\rVert_\mathbb{E}
			\end{bmatrix} \\
			&+   \mathbf{M} ^t \begin{bmatrix}
				0 \\
				a_1 	\sigma
			\end{bmatrix} \\
			=&  \sqrt{a_t}\frac{L+\mu}{1-a\mu} \sum_{\tau=0}^{t-1}\left(\mathbf{M}\sqrt{1-a\mu}\right)^{t-\tau-1}\begin{bmatrix}
				0\\
				\sqrt{a_{\tau+1}}\lVert \Delta{\bf y}^{(\tau)}\rVert_\mathbb{E}
			\end{bmatrix}\\
			& + \mathbf{M}^t \begin{bmatrix}
				0\\
				a_1 	\sigma
			\end{bmatrix}.
		\end{split}
	\end{equation*}
	Recall that the eigenvalues for matrix $\mathbf{M}$ are $\lambda_1={(\xi_1+ \xi_2)}/{2}$ and $\lambda_2={(\xi_1- \xi_2)}/{2}$,
	where \begin{equation*}
		\begin{split}
			\xi_1= \frac{\beta(2+aL)}{1-a\mu}, \quad
			\xi_2 = \frac{a(L+\mu)}{1-a\mu}\sqrt{    \frac{\beta^2L^2}{(L+\mu)^2} +\frac{4\beta(\beta+1)}{a(L+\mu)} }.
		\end{split}
	\end{equation*}
	Thus, the
	analytical form for the $n$th power of $\mathbf{M}$ is (see, e.g.,~\cite{williams1992n})
	\begin{equation*}
		\begin{split}
			\mathbf{M}^n &=\lambda_1^n\left( \frac{\mathbf{M}- \lambda_2 I }{\lambda_1-\lambda_2} \right) +\lambda_2^n\left( \frac{\mathbf{M}- \lambda_1 I }{\lambda_2-\lambda_1} \right)  \\
			& = \lambda_1^n\left( \frac{\mathbf{M}- \lambda_2 I }{\lambda_1-\lambda_2} \right) -\lambda_2^n\left( \frac{\mathbf{M}- \lambda_1 I }{\lambda_1-\lambda_2} \right) .
		\end{split}
	\end{equation*}
	It then follows that
	\begin{equation*}
		(\mathbf{M}^n)_{12}=\frac{ \mathbf{M}_{12}(\lambda_1^n  -\lambda_2^n)}{\lambda_1-\lambda_2} = \frac{ \beta(\lambda_1^n  -\lambda_2^n)}{\lambda_1-\lambda_2} \leq \frac{2\beta(\rho(\mathbf{M}))^n}{\xi_2},
	\end{equation*} 
	where $\rho(\mathbf{M})$ is the spectral radius of $\mathbf{M}$. Due to our assumption that
	\begin{equation*}
		\frac{1}{a}>\frac{\beta(2L+3\mu)}{(1-\beta)^2}+\mu>\frac{2\beta(L+\mu)}{(1-\beta)^2}
	\end{equation*}
	and $\beta\in (0,1)$,
	we have 
	$
		\xi_2>\frac{2a\beta(L+\mu)}{1-a\mu}.
	$
	Therefore,
	\begin{equation}\label{consensus_error_mid}
		\begin{split}
			\lVert	\tilde{\bf z}^{(t)} \rVert_\mathbb{E}  
			\leq&
			\frac{2\beta{\frac{L+\mu}{1-a\mu}}\sqrt{a_t}}{\xi_2}\sum_{\tau=0}^{t-1}\nu^{t-\tau-1}\sqrt{a_{\tau+1}}	\lVert \Delta{\bf y}^{(\tau)}\rVert_\mathbb{E}\\
			&+ 	\frac{2\beta}{\xi_2} \left( \rho(\mathbf{M}) \right)^t a_1\sigma\\
			\leq&  \frac{\sqrt{a_{t}}}{a}\sum_{\tau=0}^{t-1}\nu^{t-\tau-1}\sqrt{a_{\tau+1}}	 \lVert \Delta{\bf y}^{(\tau)}\rVert_\mathbb{E}\\
			&+ \frac{1-a\mu}{a(L+\mu)}\left( \rho(\mathbf{M}) \right)^ta_1\sigma.
		\end{split}
	\end{equation}
	This bound, together with Lemma~\ref{gamma_continuity}, yields 
	\begin{equation}\label{x-y}
		\begin{split}
			&\sqrt{a_{t+1}}\lVert{{\bf x}}^{(t)}-{{\bf y}}^{(t)}\rVert_\mathbb{E} \\
			&\leq  \frac{\sqrt{a_{t+1}}}{1+\mu A_t} \left(\frac{\sqrt{a_{t}}}{a}\sum_{\tau=0}^{t-1}\nu^{t-\tau-1}\sqrt{a_{\tau+1}}\lVert \Delta{{\bf y}}^{(\tau)}\rVert_\mathbb{E}+  \frac{\left( \rho(\mathbf{M}) \right)^t}{L+\mu} \sigma\right) \\
			&\leq   \frac{\sum_{\tau=0}^{t-1}\nu^{t-\tau-1}\sqrt{a_{\tau+1}}\lVert \Delta{\bf y}^{(\tau)}\rVert_\mathbb{E}}{\sqrt{1-a\mu}} +\frac{\sqrt{\frac{a}{1-a\mu}}\nu^{t}\sigma}{L+\mu}.
		\end{split}
	\end{equation} 
	The desired inequality~\eqref{consensus_error} then follows from this and Lemma~\ref{pertubed_consensus}.
\end{proof}

\subsection{Proof of Lemma~\ref{lem:descent}}\label{sec:proof-Lemma3}

\begin{proof}[Proof of Lemma~\ref{lem:descent}]
	Define 
	\begin{equation*}
		\begin{split}
			m_{t}({x}):= \sum_{\tau=0}^{t-1}a_{\tau+1}\left( \langle \overline{g}^{(\tau)},x  \rangle+ \frac{\mu}{2}\lVert x-\overline{x}^{(\tau)}\rVert^2+h(x)\right) +d({x})
		\end{split}       
	\end{equation*}
	where $m_0(x)=d(x)$.
	Due to $\overline{z}^{(t)}= \sum_{\tau=0}^{t-1}a_{\tau+1}(\overline{g}^{(\tau)}-\mu\overline{x}^{(\tau)})$ in Lemma~\ref{conservation_property},
	we can equivalently express~\eqref{eq:def-y} as
	\begin{equation*}
		\begin{split}
			y^{(t)}
			= \argmin_{x\in\mathbb{R}^m}m_t(x).
		\end{split}
	\end{equation*}
	Since $m_{\tau-1}(x)$ is strongly convex with modulus $1+\mu A_{\tau-1}$, we have, $\forall x\in \mbox{dom}(h)$, 
	\begin{equation*}
		m_{\tau-1}(x)-m_{\tau-1}(y^{(\tau-1)}) \geq \frac{1}{2}(1+\mu A_{\tau-1})\lVert x-y^{(\tau-1)}\rVert^2.
	\end{equation*}
	Further, by noticing
	\begin{small}
	\begin{equation*}
		\begin{split}
			m_\tau({x}) 
			= m_{\tau-1}({x})+a_{\tau}\left(\langle \overline{g}^{(\tau-1)},x \rangle+\frac{\mu}{2}\lVert x-\overline{x}^{(\tau-1)}\rVert^2 +h(x)\right),
		\end{split}
	\end{equation*}
\end{small}
	we have
	\begin{footnotesize}
	\begin{equation*}
		\begin{split}
			0 \leq &m_{\tau-1}(y^{(\tau)})-m_{\tau-1}(y^{(\tau-1)})-\frac{1}{2}\left(1+\mu A_{\tau-1}\right)\lVert  {y}^{(\tau)}-{ y}^{(\tau-1)} \rVert^2\\
			= &m_\tau(y^{(\tau)})-a_{\tau} \left ( \langle \overline{g}^{(\tau-1)},y^{(\tau)}\rangle+\frac{\mu}{2}\lVert y^{(\tau)}-\overline{x}^{(\tau-1)}\rVert^2+h(y^{(\tau)}) \right)\\
			&-m_{\tau-1}(y^{(\tau-1)})-\frac{1}{2}\left(1+\mu  A_{\tau-1}\right)\lVert  {y}^{(\tau)}-{ y}^{(\tau-1)}\rVert^2,
		\end{split}
	\end{equation*}
\end{footnotesize}
	%
	%
	%
	%
	%
	%
	which is equivalent to
	\begin{equation*}
		\begin{split}
			&	a_{\tau}\left( \langle  \overline{g}^{(\tau-1)},y^{(\tau)} \rangle +h(y^{(\tau)})\right)\\
			&\leq  m_\tau(y^{(\tau)}) -m_{\tau-1}(y^{(\tau-1)})-\frac{\mu}{2}a_\tau\lVert y^{(\tau)}-\overline{x}^{(\tau-1)} \rVert^2\\
			&\quad -\frac{1}{2}\Big(1+\mu  A_{\tau-1}\Big)\lVert  {y}^{(\tau)}-{ y}^{(\tau-1)} \rVert^2 .
		\end{split}
	\end{equation*}
	Summing up the above inequality from $\tau = 1$ to $\tau = t$ leads to
	\begin{equation}\label{inexact_gradient_result1}
		\begin{split}
			&	\sum_{\tau=1}^{t}a_{\tau}\left( \langle  \overline{g}^{(\tau-1)},y^{(\tau)} \rangle +h(y^{(\tau)})\right)\\
			&\leq  m_{t}(y^{(t)}) -m_{0}(y^{(0)})  -\sum_{\tau=1}^{t }\frac{1}{2}\Big((1+\mu  A_{\tau-1})\lVert  \Delta{ y}^{(\tau-1)} \rVert^2 \\
			&\quad +\mu a_\tau\lVert  y^{(\tau)}-\overline{x}^{(\tau-1)}\rVert^2  \Big) \\
			&=  m_{t}(y^{(t)})-\sum_{\tau=1}^{t }\frac{1}{2}\Big((1+\mu  A_{\tau-1})\lVert  {y}^{(\tau)}-{ y}^{(\tau-1)} \rVert^2 \\
			&\quad+\mu a_\tau\lVert  y^{(\tau)}-\overline{x}^{(\tau-1)}\rVert^2  \Big).
		\end{split}
	\end{equation}
	where the equality is due to $y^{(0)}=x^{(0)}$ and~\eqref{eq:initial-cond}.
	Then, we turn to consider
\begin{small}
	\begin{equation}\label{ine:optimum}
		\begin{split}
			&\sum_{\tau=1}^{t}a_{\tau}\langle  \nabla f(x^{(\tau-1)}),-x^* \rangle\\
			 &\leq  \max_{{x}\in\mathbb{R}^m} \Big\{  \sum_{\tau=1}^{t} a_{\tau}\Big(\langle \nabla f(x^{(\tau-1)}),-x \rangle -  \frac{\mu}{2}\lVert x-{x}^{(\tau-1)}\rVert^2\\
			 &\quad \quad \quad \quad \quad \quad \qquad-h(x)\Big) -d(x)\Big\}\\
			&\quad +d(x^*) + \sum_{\tau=1}^{t} a_{\tau} \left(\frac{\mu}{2}\lVert {x}^{(\tau-1)}-x^*\rVert^2+h(x^*)\right) \\
			&= -\min_{{x}\in\mathbb{R}^m} \Big\{  \sum_{\tau=1}^{t} a_{\tau} \Big(\langle  \nabla f(x^{(\tau-1)}),x \rangle+\frac{\mu}{2}\lVert x-{x}^{(\tau-1)}\rVert^2\\
			& \quad \quad \quad \quad \quad \quad \qquad \,\,\,\, +h(x)\Big)+d(x)\Big\}\\
			&\quad +d(x^*) + \sum_{\tau=1}^{t} a_{k}\left( \frac{\mu}{2}\lVert{x}^{(k-1)}- x^*\rVert^2 +h(x^*)\right) \\
			&= -r_{t}(x^{(t)})+d(x^*)+  \sum_{\tau=1}^{t} a_{\tau}\left( \frac{\mu}{2}\lVert {x}^{(\tau-1)}-x^*\rVert^2 +h(x^*)\right).
		\end{split}
	\end{equation}
\end{small}
	which in conjunction with~\eqref{inexact_gradient_result1} leads to the inequality in~\eqref{lem:inexact-DA}. 
\end{proof}

\section{Estimation of $\overline{a}$}\label{sec:est-a}
As we remarked after Theorem~\ref{inexact_thm}, there exists an $\bar{a}\in(0,\mu^{-1})$ such that the conditions in Theorem~\ref{inexact_thm} are satisfied by any $a\in(0,\bar{a})$. Of course, we would like to find an $\bar{a}$ as large as possible, but finding the maximum value of $\bar{a}$ requires solving a nonlinear equation associated with~\eqref{2ndcondition}, which does not admit a closed-form solution. Instead, we provide in the following lemma a conservative estimation of $\bar{a}$. 

In Theorem~\ref{inexact_thm}, we have two conditions on $a$, i.e., 
\begin{equation}
	\label{eq:111}
	\frac{1}{a} > \frac{\beta(2L + 3\mu)}{(1 - \beta)^2} + \mu,
\end{equation}
and
\begin{align}
	\frac{1}{a} > 2L - {\mu} + \frac{4L-2\mu}{\eta(a)}, \label{eq:222}
\end{align}
where $\eta(a) = (1 - a\mu)(1 - \rho(\mathbf{M})\sqrt{1 - a\mu})^2$. Moreover, $\rho(\mathbf{M}) = \lambda_1 = (\xi_1 + \xi_2)/2$, where $\xi_1,\xi_2$ are defined in~\eqref{def:xi}. Then, one can verify that by taking $a = 1/(2\mu)$, we have
\begin{equation*}
	\begin{split}
	 &\eta\left(\frac{1}{2\mu}\right) \\
	 &= \frac{\left(1- \beta(\sqrt{2}+\frac{L}{2\sqrt{2}\mu})-\sqrt{\frac{\beta^2L^2}{8\mu^2}+\beta(\beta+1)(1+\frac{L}{\mu})}\right)^2}{2}. 
 \end{split}
\end{equation*}
We have shown that $\eta$ decreases with $a$ if~\eqref{eq:111} is satisfied, so
$$ \eta(a) > \eta\left(\frac{1}{2\mu}\right) $$
for all $a$ satisfying $0<a<1/(2\mu)$ and~\eqref{eq:111}. Then, as long as $a$ satisfies
\begin{footnotesize}
\begin{equation}
	\begin{aligned}
		\frac{1}{a} & > \frac{\beta(2L + 3\mu)}{(1 - \beta)^2} + \mu, \\
		\frac{1}{a} & > 2L - {\mu} + \frac{4L-2\mu}{\eta(\frac{1}{2\mu})} \\
		&  = 2L-\mu + \frac{8L-4\mu}{\left(1- \beta(\sqrt{2}+\frac{L}{2\sqrt{2}\mu})-\sqrt{\frac{\beta^2L^2}{8\mu^2}+\beta(\beta+1)(1+\frac{L}{\mu})}\right)^2}, \\
		\frac{1}{a} & > 2\mu,
	\end{aligned}
\end{equation}
\end{footnotesize}
then $a$ also satisfies~\eqref{eq:111} and~\eqref{eq:222}. This implies that we can take
\begin{footnotesize}
	\begin{equation*}
		\begin{split}
		\overline{a} = \min \left\{  \frac{1}{2\mu} , \frac{1}{\frac{\beta(2L + 3\mu)}{(1 - \beta)^2} + \mu},  E\right\}
	\end{split}
	\end{equation*}
\end{footnotesize}
where 
\begin{footnotesize}
	\begin{equation*}
	\begin{split}
&E=\\
& \frac{\left(1- \beta(\sqrt{2}+\frac{L}{2\sqrt{2}\mu})-\sqrt{\frac{\beta^2L^2}{8\mu^2}+\beta(\beta+1)(1+\frac{L}{\mu})}\right)^2}{(2L-\mu)\left(  {4+\left(1- \beta(\sqrt{2}+\frac{L}{2\sqrt{2}\mu})-\sqrt{\frac{\beta^2L^2}{8\mu^2}+\beta(\beta+1)(1+\frac{L}{\mu})}\right)^2}\right)}.
	\end{split}
\end{equation*}
\end{footnotesize}
It would be interesting to estimate the order of $\bar{a}$ when the condition number $\kappa = L/\mu$ goes to $\infty$ and the $\beta$, which relates to the connectivity of the stochastic network, goes to $1$. By the standard limiting argument, one can verify that the dominating term inside the above brace is the second term, which is in the order $\mathcal{O}((1-\beta)^2/L)$.

\section{Proof of Theorem~\ref{exact_thm}}\label{sec:proof-Theorem1}
In this section, we first provide a technical lemma, and then present the proof of Theorem~\ref{exact_thm}.
\begin{lemma}\label{dual_averaging_inequality}
	Suppose that the premise of Theorem~\ref{exact_thm} holds.	For the sequence $\{x^{(t)}\}_{t\geq 0}$ generated by~\eqref{eq:standard-DA}, it holds that
	\begin{equation}\label{exact}
		\begin{split}
			&\sum_{\tau=1}^{t}a_{\tau}\left( \langle  \nabla f(x^{(\tau-1)}),x^{(\tau)}-x^* \rangle +  h(x^{(\tau)})-h(x^*)  \right)-d(x^*) \\
			& \leq 
			-\frac{1}{2}\sum_{\tau=1}^{t }\big(1+\mu  A_{\tau}\big)\lVert x^{(\tau)}-{ x}^{(\tau-1)} \rVert^2    \\
			& \quad  -\frac{\mu}{2}\sum_{\tau=1}^{t} a_{\tau}\lVert {x}^{(\tau-1)}-x^*\rVert^2.
		\end{split}
	\end{equation}
\end{lemma}   
\begin{proof}
	We define
	\begin{equation*}
		\begin{split}
			r_{t}({x}):= \sum_{\tau = 0}^{t - 1} a_{\tau + 1} \ell(x; x^{(\tau)}) + d(x), \quad t = 0, 1, \dots,
		\end{split}       
	\end{equation*}
	where $r_0(x)=d(x)$ and $\ell(x;x^{(\tau)})$ is defined in~\eqref{eq:def-ell}.
	It then follows that for any $\tau \geq 1$,
	\begin{equation}\label{eq:def-r}
		\begin{split}
		r_\tau(x) &= r_{\tau-1}({x})\\
		&\quad +a_{\tau}\left(\langle \nabla f(x^{(\tau-1)}),x \rangle+\frac{\mu}{2}\lVert x-{x}^{(\tau-1)}\rVert^2 +h(x)\right).
	\end{split}
	\end{equation}
	By~\eqref{eq:standard-DA}, we know that $x^{(\tau-1)} = \argmin_{x\in\mathbb{R}^m} r_{\tau-1}(x)$.
	Moreover, $r_{\tau-1}(x)$ is strongly convex with modulus $1+\mu A_{\tau-1}$. Then,
	we obtain $\forall x\in \mbox{dom}(h)$
	\begin{equation*}
		r_{\tau-1}(x)-r_{\tau-1}(x^{(\tau-1)}) \geq \frac{1}{2}(1+\mu A_{\tau-1})\lVert x-x^{(\tau-1)}\rVert^2.
	\end{equation*}
	Therefore,
	\begin{footnotesize}
	\begin{equation*}
		\begin{split}
			0  \leq& r_{\tau-1}(x^{(\tau)})-r_{\tau-1}(x^{(\tau-1)})-\frac{1}{2}\big(1+\mu A_{\tau-1}\big)\lVert  x^{(\tau)}-{ x}^{(\tau-1)} \rVert^2\\
			=& -a_{\tau} \left ( \langle \nabla f(x^{(\tau-1)}),x^{(\tau)} \rangle+\frac{\mu}{2}\lVert x^{(\tau)}-{x}^{(\tau-1)}\rVert^2+h(x^{(\tau)}) \right) \\
			&+r_\tau(x^{(\tau)})-r_{\tau-1}(x^{(\tau-1)}) -\frac{1}{2}\big(1+\mu  A_{\tau-1}\big)\lVert  x^{(\tau)}-{ x}^{(\tau-1)}\rVert^2,
		\end{split}
	\end{equation*}
\end{footnotesize}
	where the equality follows from~\eqref{eq:def-r}.
	%
	%
	%
	%
	%
	%
	This, together with $A_{\tau} = A_{\tau-1}+a_{\tau}$, leads to
	\begin{equation*}
		\begin{split}
		&	a_{\tau}  \left( \langle  \nabla f(x^{(\tau-1)}),x^{(\tau)} \rangle + h(x^{(\tau)})\right) \\
		&	\leq  r_\tau(x^{(\tau)}) -r_{\tau-1}(x^{(\tau-1)})-\frac{1}{2}\big(1+\mu  A_{\tau}\big)\lVert x^{(\tau)}-{ x}^{(\tau-1)} \rVert^2 .
		\end{split}
	\end{equation*}
	Summing up the above inequality from $\tau = 1$ to $\tau = t$ yields
	\begin{equation}\label{exact_gradient_result1}
		\begin{split}
			&\sum_{\tau=1}^{t}a_{\tau}\left( \langle  \nabla f(x^{(\tau-1)}),x^{(\tau)} \rangle + h(x^{(\tau)})\right)  \\
			&\leq r_{t}(x^{(t)}) -r_{0}(x^{(0)}) -\sum_{\tau=1}^{t }\frac{1}{2}\big(1+\mu  A_{\tau}\big)\lVert  x^{(\tau)}-{ x}^{(\tau-1)} \rVert^2   \\
			&=  r_{t}(x^{(t)}) -\sum_{\tau=1}^{t }\frac{1}{2}\big(1+\mu  A_{\tau}\big)\lVert x^{(\tau)}-{ x}^{(\tau-1)} \rVert^2,
		\end{split}
	\end{equation}
	where the equality follows from $r_0(x) = d(x)$ and~\eqref{eq:initial-cond}.
		Then, we turn to consider
		\begin{scriptsize}
			\begin{equation}\label{ine:optimum}
				\begin{split}
				&	\sum_{\tau=1}^{t}a_{\tau}\langle  \nabla f(x^{(\tau-1)}),-x^* \rangle\\
					& \leq \max_{{x}} \left\{  \sum_{\tau=1}^{t} a_{\tau}\left(\langle \nabla f(x^{(\tau-1)}),-x \rangle -  \frac{\mu}{2}\lVert x-{x}^{(\tau-1)}\rVert^2-h(x)\right) -d(x)\right\}\\
					&\quad +d(x^*) + \sum_{\tau=1}^{t} a_{\tau} \left(\frac{\mu}{2}\lVert {x}^{(\tau-1)}-x^*\rVert^2+h(x^*)\right) \\
					&=-\min_{{x}\in\mathbb{R}^m} \left\{  \sum_{\tau=1}^{t} a_{\tau} \left(\langle  \nabla f(x^{(\tau-1)}),x \rangle+\frac{\mu}{2}\lVert x-{x}^{(\tau-1)}\rVert^2+h(x)\right)+d(x)\right\}\\
					&\quad +d(x^*) + \sum_{\tau=1}^{t} a_{\tau}\left( \frac{\mu}{2}\lVert{x}^{(\tau-1)}- x^*\rVert^2 +h(x^*)\right) \\
					&= -r_{t}(x^{(t)})+d(x^*)+  \sum_{\tau=1}^{t} a_{\tau}\left( \frac{\mu}{2}\lVert {x}^{(\tau-1)}-x^*\rVert^2 +h(x^*)\right).
				\end{split}
			\end{equation}
		\end{scriptsize}
		Upon summing up~\eqref{exact_gradient_result1} and the above inequality, we obtain~\eqref{exact} as desired.
	\end{proof}

	\begin{proof}[Proof of Theorem~\ref{exact_thm}]
		
		Recall that $f = \frac{1}{n}\sum_{i=1}^nf_i$. Using~\eqref{eq:Lip-cond} and~\eqref{eq:strongly-convex} sequentially, we have
		\begin{equation*}
			\begin{split}
				& a_\tau\left( f({x}^{(\tau)}) -f(x^*)\right)\\
				& \leq a_\tau\Big( \frac{L}{2}\lVert {x}^{(\tau)}-x^{(\tau-1)} \rVert^2+ f(x^{(\tau-1)})\\
				& \quad \quad \quad +\langle \nabla f(x^{(\tau-1)}),{x}^{(\tau)}-x^{(\tau-1)} \rangle-f(x^*) \Big)\\
				& \leq  a_\tau\Big(\frac{L}{2}\lVert {x}^{(\tau)}-x^{(\tau-1)} \rVert^2 +\Big\langle \nabla f(x^{(\tau-1)}),{x}^{(\tau)}-x^* \Big\rangle\\
				&\quad \quad \quad  -\frac{\mu}{2}\lVert x^{(\tau-1)}-x^* \rVert^2 \Big).
			\end{split}
		\end{equation*}
		Upon summing up the above inequality from $\tau = 1$ to $\tau = t$ and using Lemma~\ref{dual_averaging_inequality} and $F =f + h$, 
		we obtain
		\begin{equation*}
			\begin{split}
				&	\sum_{\tau=1}^{t}a_{\tau} \left(  F(x^{(\tau)}) -F(x^*)  \right)  \\
				&	\leq   \sum_{\tau=1}^{t} \Big(  \frac{a_\tau}{2}\Big( L \lVert x^{(\tau)}-x^{(\tau-1)} \rVert^2-\frac{1+\mu A_{\tau}}{a_\tau}  \lVert x^{(\tau)}-{ x}^{(\tau-1)} \rVert^2
				\Big) \Big)\\
				& \quad +d(x^*).
			\end{split}
		\end{equation*}
			According to~\eqref{eq:at} and $A_t = \sum_{\tau=1}^ta_\tau$, one has
			\begin{equation}
				\label{relation_Atoa}
				\frac{1+\mu A_\tau}{a_\tau}= \frac{(\frac{1}{1-a\mu})^\tau}{\frac{a}{1-a\mu}\Big( \frac{1}{1-a\mu}\Big)^{\tau-1}}=\frac{1}{a}.
			\end{equation}
			By substituting this into the above inequality and using the condition $a\leq L^{-1}$, we obtain
			\begin{equation*}
				\begin{split}
				&\sum_{\tau=1}^{t}a_{\tau} \left(  F(x^{(\tau)}) -F(x^*)  \right)  \\
					&\leq  \left( L - \frac{1}{a}\right) \sum_{\tau=1}^{t} \frac{a_\tau}{2}\lVert x^{(\tau)}-x^{(\tau-1)} \rVert^2 +d(x^*) \leq d(x^*).
				\end{split}
			\end{equation*} 
			Upon dividing both sides of the above inequality by $A_t$ and using the convexity of $F$ and $\tilde{x}^{(t)} = A_t^{-1}\sum_{\tau=1}^t a_\tau x^{(\tau)}$, we obtain
			\begin{equation*}
				\begin{split}
					F(\tilde{x}^{(t)}) -F(x^*)
					\leq   \frac{ d(x^*)}{A_t }.
				\end{split}
			\end{equation*}  
			Now it remains to show the statements i) and ii) in Theorem~\ref{exact_thm}. By the definitions of $a_t$ and $A_t$, we readily have $A_t = at$ when $\mu=0$ and
			$$ \frac{1}{A_t} = \frac{\mu}{(\frac{1}{1-a\mu})^t-1} $$
			when $\mu>0$. Moreover, by $0 < a < L^{-1}$ and $L \geq \mu$, one has $0 < a\mu < 1$ when $\mu > 0$. This, together with the above identity, yields that when $\mu > 0$, 
			\begin{equation*}
				\begin{split}
			&	\frac{1}{A_t} = \frac{\mu}{(\frac{1}{1-a\mu})^t-1}  = \frac{\mu(1 - a\mu)^t}{1 - (1 - a\mu)^t}\\
				& \leq  \frac{\mu(1 - a\mu)^t}{1 - (1 - a\mu)} =  \frac{(1 - a\mu)^t}{a}.
			\end{split}
			\end{equation*} 
			This completes the proof.
		\end{proof}

\bibliographystyle{IEEEtran}
\bibliography{barejrnl}

\end{document}